\documentclass[12pt]{amsart}%

\usepackage{graphicx}
\usepackage{tikz-cd}
\usepackage{mathtools}
\usepackage{wasysym}

\usepackage{relsize}
\usepackage{enumerate}
\usepackage{multicol,multirow}
\usepackage{amsmath,amssymb,amsfonts}
\usepackage{latexsym}
\usepackage{mathrsfs}
\usepackage{amsthm}
\usepackage{rotating}
\usepackage{appendix}
\usepackage[numbers]{natbib}
\usepackage{ifpdf}
\usepackage[T1]{fontenc}
\usepackage{textcomp}
\usepackage[a4paper,body={16.3cm,22.8cm},centering]{geometry}
\usepackage[colorlinks]{hyperref}
\usepackage{tikz}
\usepackage{orcidlink}
\usepackage{stmaryrd}
\usepackage{xypic}

\begin{document}











\usetikzlibrary{calc,shapes.geometric}
\tikzset{
baseon/.style={baseline={($(#1)+(0,-0.58ex)$)}},
baseon/.default=current bounding box.center,
every picture/.style=baseon,
lst/.style={},
dst/.style={circle,inner sep=1pt,outer sep=0pt,fill,draw,dst2},
dst2/.style={fill=white},
ddst/.style={diamond,draw,inner sep=1pt},
eest/.style={ellipse,draw,inner sep=1pt,minimum size=2ex},
}


\font \eightrm=cmr8
\font \sevenrm=cmr7
\font \fiverm=cmr5
\newcommand{\nc}{\newcommand}
\nc\smsc{0.8}
\def\tthreeone{\XX[scale=\smsc]{\xxr{-5}5}}
\def\tthreetwo{\XX[scale=\smsc]{\xxl55}}
\def\tfourone{\XX[scale=\smsc]{\xxr{-4}4\xxr{-7.5}{7.5}}}
\def\tfourtwo{\XX[scale=\smsc]{\xxr{-5}5\xxl{-2}8}}
\def\tfourthree{\XX[scale=\smsc]{\xxl{4}4\xxl{7.5}{7.5}}}
\def\tfourfour{\XX[scale=\smsc]{\xxl{5}5\xxr{2}8}}
\def\tfourfive{\XX[scale=\smsc]{\xxr{-6}6\xxl66}}
\def\oprec{\!\!\joinrel{\ocircle\hskip -12.5pt \prec}\,}
\def\soprec{\,\joinrel{\ocircle\hskip -6.7pt \prec}\,}
\def \endsquare{$\sqcup \!\!\!\! \sqcap$ \\}
\def\diagramme #1{\vskip 4mm \centerline {#1} \vskip 4mm}
\def\att#1{\textcolor{red}{!!!}\textcolor{blue}{#1}}
\nc{\ooverline}[1]{\overline{\overline #1}}
\nc{\uunderline}[1]{\underline{\underline #1}}
\def \restr#1{\mathstrut_{ | \, \scriptstyle{ #1}}}
\def \srestr#1{\mathstrut_{\scriptstyle |}\hbox to
-1.5pt{}\raise-4pt\hbox{$\hskip 1pt\scriptscriptstyle #1$}}
\nc{\mop}[1]{\mathop{\hbox {\rm #1} }\nolimits}
\nc{\gmop}[1]{\mathop{\hbox {\bf #1} }\nolimits}
\def\starz{{\displaystyle\mathop{\star}\limits}_z}
\def\astT{{\displaystyle\mathop{\ast}\limits}_T}
\nc{\smop}[1]{\mathop{\hbox {\sevenrm #1} }\nolimits}
\nc{\ssmop}[1]{\mathop{\hbox {\fiverm #1} }\nolimits}
\nc{\mopl}[1]{\mathop{\hbox {\rm #1} }\limits}
\def\dbar{d\hskip-3pt \raise 4pt\hbox{-}}
\nc{\smopl}[1]{\mathop{\hbox {\sevenrm #1} }\limits}
\nc{\ssmopl}[1]{\mathop{\hbox {\fiverm #1} }\limits}
\renewcommand\baselinestretch{1}

\newcommand{\delete}[1]{}

\nc{\mlabel}[1]{\label{#1}}  
\nc{\mcite}[1]{\cite{#1}}  
\nc{\mref}[1]{\ref{#1}}  
\nc{\mbibitem}[1]{\bibitem{#1}} 

\delete{
\nc{\mlabel}[1]{\label{#1}  
{\hfill \hspace{1cm}{\small\tt{{\ }\hfill(#1)}}}}
\nc{\mcite}[1]{\cite{#1}{\small{\tt{{\ }(#1)}}}}  
\nc{\mref}[1]{\ref{#1}{{\tt{{\ }(#1)}}}}  
\nc{\mbibitem}[1]{\bibitem[\bf #1]{#1}} 
}

\newtheorem{theorem}{Theorem}[section]

\newtheorem{notation}[theorem]{Notation}
\newtheorem{corollary}[theorem]{Corollary}
\newtheorem{proposition}[theorem]{Proposition}
\newtheorem{lemma}[theorem]{Lemma}

{\theoremstyle{definition}
\newtheorem{remark}[theorem]{Remark}
\newtheorem{example}[theorem]{Example}
\newtheorem{conjecture}[theorem]{Conjecture}
\numberwithin{equation}{section}
\newtheorem{fact}[theorem]{Fact}

\newtheorem{problem}[theorem]{Problem}
\newtheorem{definition}[theorem]{Definition}
\newtheorem{definitions}[theorem]{Definitions}
}

\renewcommand{\labelenumi}{{\rm \alph{enumi}}}
\renewcommand{\theenumi}{\alph{enumi}}

\renewcommand{\labelenumii}{{\rm \roman{enumii}}}
\renewcommand{\theenumii}{\roman{enumii}}
\newcommand\alphlist{a,b,c,d,e,f,g,h,i,j,k,l,m,n,o,p,q,r,s,t,u,v,w,x,y,z}
\newcommand\Alphlist{A,B,C,D,E,F,G,H,I,J,K,L,M,N,O,P,Q,R,S,T,U,V,W,X,Y,Z}
\newcommand\getcmds[3]{\expandafter\newcommand\csname #2#1\endcsname{#3{#1}}}
\makeatletter
\@for\x:=\alphlist\do{\expandafter\getcmds\expandafter{\x}{frak}{\mathfrak}}
\@for\x:=\Alphlist\do{\expandafter\getcmds\expandafter{\x}{frak}{\mathfrak}}
\makeatother

\nc{\bfk}{{\bf k}}
%
\nc{\sha}{\shuffle}

\nc{\id}{\mathrm{id}}
\nc{\Id}{\mathrm{Id}}
\nc{\lbar}[1]{\overline{#1}}
\nc{\ot}{\otimes}
\nc{\dep}{\mathrm{dep}}
\nc{\ver}{\mathrm{ver}}

\nc{\tred}[1]{\textcolor{red}{#1}} \nc{\tgreen}[1]{\textcolor{green}{#1}}
\nc{\tblue}[1]{\textcolor{blue}{#1}} \nc{\tpurple}[1]{\textcolor{purple}{#1}}
\nc{\tcyan}[1]{\textcolor{cyan}{#1}} 
\nc{\tblk}[1]{\textcolor{black}{#1}}

\nc{\li}[1]{\tpurple{\underline{Li:}#1 }}
\nc{\liadd}[1]{\tpurple{#1}}
\nc{\xing}[1]{\tblue{\underline{Xing:}#1 }}
\nc{\yuan}[1]{\tred{\underline{Yuan:}#1 }}
\nc{\markus}[1]{\tred{\underline{Markus:} #1}}
\nc{\dominique}[1]{\tpurple{\underline{Dominique: }#1 }}
\long\def\ignore#1{}

\newcommand\stset[2]{\tikzset{#1/.style={#2}}}
\newcommand\stadd[2]{\tikzset{#1/.append style={#2}}}
\newcommand\dadd[1]{\stset{dst2}{#1}}
\newcommand\treeo[2][]{\tikz[x=0.7cm,y=0.7cm,line width=0.15ex,
every node/.style={font=\scriptsize,inner sep=1pt,label distance=1pt},#1]{%
\coordinate (o) at (0,0);#2}}%
\newcommand\treeoo[2][]{\treeo[#1]{\path (o) node[dst,name=o]{};#2}}%
\newcommand\cddf[3]{%
\coordinate (#2) at ($(#1)+(#3)$);
\draw[lst] (#1)--(#2);
\node[dst] at (#1) {};}
\newcommand\cdx[4][1]{\cddf{#2}{#3}{#4:#1*0.5cm}}
\newcommand\cdl[2][1]{\cdx[#1]{#2}{#2l}{135}}
\newcommand\cdr[2][1]{\cdx[#1]{#2}{#2r}{45}}
\newcommand\cdlr[2][1]{%
\foreach \i in {#2} {\cdl[#1]{\i}\cdr[#1]{\i}}}
\newcommand\cdlrx[3][1]{%
\cdx[#1]{#2}{#2l}{180-#3}\cdx[#1]{#2}{#2r}{#3}}
\newcommand\cda[2][1]{\cdx[#1]{#2}{#2a}{90}}
\newcommand\cdb[2][1]{\cdx[#1]{#2}{#2b}{-90}}
\def\zzz#1`#2...#3`#4...#5`#6@{%
--++(#1)
node[dst,label={#5:$#6$},name=#2]{}
node[midway,auto,#3]{$#4$}
}
\def\ddd#1`#2`#3@{+(#1)node[ddst,name=#2]{$#3$}}
\def\eee#1`#2`#3@{+(#1)node[eest,name=#2]{$#3$}}
\def\xxx#1`#2@{node[midway,auto,inner sep=1pt,#1]{$#2$}}
\def\pp#1`#2`#3@{node[dst,label={#2:$#3$},pos=#1]{}}
\def\oo#1`#2`#3@{\path (o) node[dst,label={#2:$#3$},name=o,#1]{};}
\def\eoo#1`#2@{\node[eest,name=o,#1] at (o) {$#2$};}

\newif\ifshowjdq
\showjdqtrue
\newcommand\setXXclip[3]{%
\def\XXheight{#1}\def\XXdepth{#2}\def\XXwidth{#3}}
\setXXclip{1}{-0.5}{1.1}
\newcommand\scopeclip[1]{\begin{scope}
\clip(-\XXwidth,\XXdepth)rectangle(\XXwidth,\XXheight);#1
\end{scope}}
\newcommand\XX[2][]{%
\tikz[x=0.5cm,y=0.5cm,baseline,inner sep=1.5pt,line width=0.15ex,
every node/.style={font=\scriptsize},#1]{
\scopeclip{\draw (2,2)--(0,0)--(-2,2) (0,-2)--(0,0);
\ifshowjdq\node[dst]at(0,0){};\fi}#2}}
\newcommand\xx[3]{%
\scopeclip{\draw(#1/10,#2/10)--+(#3*45:3);
\ifshowjdq\node[dst]at(#1/10,#2/10){};\fi}}
\newcommand\xxl[2]{\xx{#1}{#2}3}
\newcommand\xxr[2]{\xx{#1}{#2}1}
\newcommand\xxlr[2]{\xxl{#1}{#2}\xxr{#1}{#2}}
\newcommand\xxh[6]{
\draw(#1/10,#2/10)+(0.5*#3*45+0.5*#4*45:#6) node[above] {$#5$};}
\newcommand\xxhu[4][0.15]{\xxh{#2}{#3}13{#4}{#1}}
\newcommand\nodea[3][]{\node[above=1pt,#1] at (#2) {$#3$};}
\newcommand\nax[2][]{\foreach \i in {#2} {\nodea[#1]{\i,1}{x}}}
\newcommand\naxx[3][]{\foreach \i in {#2} {\nodea[#1]{\i,1}{#3}}}
\newcommand\naxxx[2][]{\foreach \i/\j in {#2} {\nodea[#1]{\i,1}{\j}}}

\makeatletter
\newcommand\simra{\mathrel{\mathpalette\@verra\sim}}
\def\@verra#1#2{\lower.5\p@\vbox{\lineskiplimit\maxdimen \lineskip-.5\p@
\ialign{$\m@th#1\hfil##\hfil$\crcr#2\crcr\rightarrow\crcr}}}
\makeatother

\nc{\dnx}{\Delta_n A} \nc{\dx}{\Delta A} \nc{\dgp}{{\rm deg_{P}}}
\nc{\dgt}{{\rm deg_{T}}} \nc{\dg}{{\rm deg}} \nc{\ida}{ID($A$)} \nc{\tu}{\tilde{u}} \nc{\tv}{\tilde{v}}
\nc{\nr}{\calr_n} \nc{\nz}{\calz_n} \nc{\fun}{\cala_{n,d}}
 \nc{\fbase}{\calb} \nc{\LF}{\mathrm{RF}} \nc{\FFA}{\mathrm{LF}} \nc{\irr}{\mathrm{Irr}}
 \nc{\result}{\bfk\mathrm{Irr}(S_n)}  \nc{\I}{I_{\mathrm{ID},n}^0}
 \nc{\nrs}{\calr_n^\star} \nc{\ii}{\mathrm{I}} \nc{\iii}{\mathrm{II}}
\nc{\intl}{{\rm int}}\nc{\ws}[1]{{#1}}\nc{\deleted}[1]{\delete{#1}}\nc{\plas}{placements\xspace}

\nc{\bim}[1]{#1}  \nc{\shaop}{\sha_{\Omega}^{+}}  \nc{\shao}{\sha_{\Omega}}
\nc{\bbim}[2]{#1 #2} \nc{\bbbim}[2]{#1,\, #2} \nc{\RBF}{{\rm RBF}}
\nc{\frb}{F_{\RB}} \nc{\shaf}{\ssha_{\tiny{\Omega}}} \nc{\sham}{\diamond_{\tiny{\Omega}}}
\nc{\lf}{\lfloor} \nc{\rf}{\rfloor} \nc{\shan}{\ssha_{\lambda}}
\nc{\rlex}{{\rm {lex}}} \nc{\bb}{\Box} \nc{\ra}{\rightarrow}
\nc{\DDF}{\mathrm{DD}(X,\,\Omega)}\nc{\DTF}{\mathrm{DT}(X,\,\Omega)} \nc{\DT}{\mathrm{DT}'(\Omega,\,V)}
\nc{\bra}{\mathrm{bra}} \nc{\bre}{\mathrm{bre}}
\nc{\dec}{\mathrm{dec}} \nc{\diamondw}{\diamond_{w}}
\nc{\type}{\mathrm{type}}


\nc\caF[1]{\cal{F}_{#1}(X,\,\Omega)}
\nc\calt{\cal{T}(X,\,\Omega)} \nc\caltn{\cal{T}_n(X,\,\Omega)}
\nc\calta{\cal{T}_0(X,\,\Omega)}
\nc\caltb{\cal{T}_1(X,\,\Omega)}
\nc\caltc{\cal{T}_2(X,\,\Omega)}
\nc\caltd{\cal{T}_3(X,\,\Omega)}
\nc\caltm{\cal{T}_m(X,\,\Omega)}
\nc\caltx{\cal{T}(X)}
\nc\calf{\cal{F}(X,\,\Omega)}
\nc\fram{\frak{M}(\Omega,\, X)}
\nc\shaw{\sha^{NC}_w(\Omega,\, X)}
\nc\dw{\diamond_w} \nc\dl{\diamond_\ell}
\nc\shal{\sha^{NC}_\ell(X,\, \Omega)} \nc\shav{\sha^{NC}_w(\Omega,\, V)} \nc\shat{\sha^{NC,1}_w(\Omega,\, T^{+}(V))}
\nc{\cfo}{\cal{F}(X,\,\Omega)}
\nc{\sh}{\rm{Sh}}
\nc{\lar}{\varinjlim}
\nc\XO{(X,\,\Omega)}
\def\cxo#1#2;{\cal{#1}#2\XO}
\nc\lrf[2]{B_{#2}^+(#1)}
\nc{\fd}{\mathrm{\text{typed angularly decorated planar rooted trees}}}
\nc{\rb}{\mathrm{RBFWs}} \nc{\dfw}{\mathrm{DFW{(X)}}} \nc{\tfw}{\mathrm{TFW{(X)}}}
\nc{\tfv}{\mathrm{TFW{(V)}}}

\def\Ve#1,#2,#3;{\vee_{#1,\,(#2,\,#3)}}
\def\bigv#1;#2;#3;{\bigvee\nolimits_{#1}^{#2;\,#3}}
\nc\rjt[2]{\mathrel{\mathop{\longrightarrow}\limits^{#1\hfill}_{\hfill#2}}}
\nc{\pl}{\cal{PLF}}
\nc{\tr}{\cal{RTF}}
\nc{\im}{\mathrm{im} }
\nc{\ff}{\cal{F}_\Omega}
\nc{\tm}{T_\Omega}
\nc{\calp}{\cal{P}}
\makeatletter
\nc\dd{\@ifnextchar'{\ddA}{\ddB}}
\def\ddA'#1;{\rhd'_{#1\,}}
\def\ddB#1;{\rhd_{#1\,}}
\nc{\pbt}{\mathrm{PBT}}
\nc{\ad}{\mathrm{ad}}

\nc{\hRR}{\hat {\mathbb R}}
\nc{\RR}{{\mathbb R}}
\nc{\NN}{{\mathbb N}}
\nc{\ol}[1]{\overline{#1}}
\nc{\set}{{\rm \bf set}}
\nc{\setx}{{\rm \bf set^{\times}}}
\nc{\setn}{{\rm \bf set_{{\mathbb N}}}}
\nc{\Sb}{{\mathbb S}}
\nc{\cp}{\Vdash}
\nc{\te}{\otimes}
\nc{\sus}{\subseteq}
\nc{\BF}{{\rm BF}\,}
\nc{\Pio}{\Pi^\circ}
\nc{\SM}{{\rm SM}\,}
\nc{\SMat}{{\rm SMat}\,}
\nc{\MMat}{{\rm MMat}\,}
\nc{\ti}{\times}
\nc{\ben}{{\bf 1}}
\nc{\bigsum}{{\mathlarger{\sum}}}
\nc{\nil}{{\mathbf 0}}
\nc{\Hom}{\text{Hom}}
\nc{\cT}{{\mathcal T}}
\newcommand{\btl}{\blacktriangleleft}
\newcommand{\wtl}{\lhd}
\newcommand{\pil}{\rightarrow}
\newcommand{\hN}{\hat{\mathbb N}}
\newcommand{\note}{\noindent {\bf Note.}}
\newcommand{\ppr}{{\prime \prime}}
\newcommand{\barA}{\overline{A}}
\newcommand{\barB}{\overline{B}}
\newcommand{\App}{A^{\prime \prime}}
\newcommand{\cC}{{\mathcal C}}
\newcommand{\cB}{{\mathcal B}}
\newcommand{\PPI}{{PP(I)}}
\newcommand{\DB}{{\Delta_{\cB}}}

\newcommand{\mto}[1]{\stackrel{#1}\longrightarrow}
\newcommand{\ZZ}{\mathbb{Z}}
\newcommand{\QQ}{\mathbb{Q}}
\newcommand{\kk}{{\Bbbk}}
\newcommand{\bH}{{\mathbf H}}


\definecolor{tol1}{HTML}{332288}
\definecolor{tol2}{HTML}{AA4499}
\definecolor{tol3}{HTML}{DDCC77}
\definecolor{tol4}{HTML}{44AA99}
\definecolor{tol5}{HTML}{88CCEE}
\definecolor{tol6}{HTML}{CC6677}
\definecolor{tol7}{HTML}{117733}
\definecolor{tol8}{HTML}{882255}


\newcommand{\Pre}{{\mathrm {Pre}}}
\newcommand{\Top}{{\mathrm {Top}}}
\newcommand{\cU}{{\mathcal U}}
\renewcommand{\cT}{{\mathcal T}}
\newcommand{\cS}{{\mathcal S}}
\newcommand{\low}{{\mathrm {low}}}
\newcommand{\bfx}{{\mathbf x}}
\newcommand{\bfu}{{\mathbf u}}
\newcommand{\hele}{{\mathbb Z}}


\newcommand{\rleftarrows}{\mathrel{\raise.75ex\hbox{\oalign{%
  $\hfil\scriptstyle\relbar$\cr
  \vrule width0pt height.5ex$\scriptstyle\smash\leftarrow$\cr}}}}
\newcommand{\rightlarrows}{\mathrel{\raise.75ex\hbox{\oalign{%
  $\scriptstyle\rightarrow$\hfil\cr
  $\scriptstyle\vrule width0pt height.5ex\relbar$\cr}}}}
\newcommand{\Rrelbar}{\mathrel{\raise.75ex\hbox{\oalign{%
  $\scriptstyle\relbar$\cr
  \vrule width0pt height.5ex$\scriptstyle\relbar$}}}}
\newcommand{\longrightleftarrows}{\rleftarrows\joinrel\Rrelbar\joinrel\rightlarrows}
\newcommand{\bihom}[2]{\, \overset{#1}{\underset{#2}{\longrightleftarrows}}\, }
\newcommand{\pow}{{\mathcal P}}
\newcommand{\ppow}{{\mathcal {PP}}}
\newcommand{\egp}{{\Pi}}
\newcommand{\pre}{{\mathrm {pre}}}
\renewcommand{\top}{{\mathrm {top}}}
\newcommand{\preo}{{\mathrm {preo}}}
\newcommand{\topl}{{\mathrm {topl}}}
\newcommand{\cpre}{\mathrm{cpre}}
\newcommand{\ctop}{\mathrm{ctop}}

\newcommand{\con}{{\mathrm {con}}}
\newcommand{\op}{{\mathrm {op}}}
\newcommand{\ua}[1]{\uparrow \!\! #1}
\newcommand{\da}[1]{\downarrow \! #1}

\newcommand{\uai}[2]{\uparrow_{\! #1} \! #2}
\newcommand{\dai}[2]{\downarrow_{\! #1} \! #2}

\newcommand{\Pcon}{\Pre}
\newcommand{\alin}{\mathrm{Alin}}
\newcommand{\alinb}{\mathrm{AlinBu}}

\newcommand{\vect}{\text{\bf vect}}
\newcommand{\minPre}{\text{minPre}}
\newcommand{\EGP}{{\rm{EGP}}}
\newcommand{\PRE}{{\rm{PRE}}}
\newcommand{\MOD}{{\rm{MOD}}}
\newcommand{\oU}{\overline{U}}

\newcommand{\fone}{{\mathbf 1}}
\newcommand{\one}{1}
\newcommand{\ldim}{\text{lind}}
\newcommand{\bfan}{{\cB}}
\newcommand{\pphat}[1]{\widehat{#1}}
\newcommand{\nfan}{\mathcal{N}}
\newcommand{\modu}{{\rm{mod}}}
\newcommand{\cone}{{\rm{cone}}}
\newcommand{\subc}{\text{Sub}}
\newcommand{\supc}{\text{Sup}}
\newcommand{\face}{\text{face}}
\newcommand{\coface}{\text{coface}}
\renewcommand{\con}{\text{cones}}
\newcommand{\pol}{\text{pol}}

\newcommand{\oV}{\overline{V}}
\newcommand{\oP}{\overline{P}}
\newcommand{\oF}{\overline{F}}
\newcommand{\iso}{\cong}

\newcommand{\Sp}{\mathsf S}
\newcommand{\Tp}{\mathsf T}
\newcommand{\Pp}{\mathsf P}
\newcommand{\vectx}{{\bf vect}^{\times}}
\newcommand{\Homf}{\text{Hom}^{\rm{fin}}}

\newcommand{\poly}{{\sf polh}}
\newcommand{\conesp}{{\sf cone}}
\newcommand{\conespa}{{\sf cone}^a}
\newcommand{\lin}{{\sf lin} }
\newcommand{\lina}{{\sf lin}^a }
\newcommand{\res}{\text{res}}

\newcommand{\hP}{\hat{P}}
\newcommand{\hF}{\hat{F}}
\newcommand{\hG}{\hat{G}}

\renewcommand{\egp}{{\sf{egp}}}
\newcommand{\egc}{{\sf{egc}}}
\newcommand{\sm}{{\sf{sm}}}
\renewcommand{\mod}{{\sf{mod}}}
\newcommand{\iunit}{E} 

\newcommand{\bool}{\mathsf{bool}^e}
\newcommand{\smod}{{\mathsf {smod}}}
\newcommand{\egperm}{{\rm{egp}}}
\newcommand{\egca}{{\mathsf {egc}}^a}

\newcommand{\teh}{\overset{H}{\te}}
\newcommand{\tec}{\overset{C}{\te}}

\newcommand{\benh}{\ben_H}
\newcommand{\benc}{\ben_C}

\newcommand{\idmon}{E} 

\newcommand{\bfnull}{{\bf 0}}
\newcommand{\disc}{\bullet}

\newcommand{\sfp}{\mathsf p}
\newcommand{\sfc}{\mathsf c}

\title[Extended generalized permutahedra and cointeracting bialgebras]
{Extended generalized permutahedra and cointeracting bialgebras}
\thispagestyle{empty}
\author{Gunnar Fl\o ystad}
\address{Matematisk institutt,
Universitetet i Bergen, Realfagbygget,
All\'egaten 41
Bergen, Norway}
\email{Gunnar.Floystad@uib.no}
\author{Dominique Manchon}
\address{Laboratoire de Math\'ematiques Blaise Pascal,
CNRS--Universit\'e Clermont-Auvergne,
3 place Vasar\'ely, CS 60026,
F63178 Aubi\`ere, France}
\email{Dominique.Manchon@uca.fr}

\keywords{Extended generalized permutahedra, Submodular functions,
  Preorders, Braid fan, Bimonoids, Bialgebras, Cointeraction }
\subjclass[2010]{16T15, 16T30, 52B05}
\maketitle

\setcounter{tocdepth}{1}

\begin{abstract}
  A Hopf monoid structure on extended generalized permutahedra (EGP) was
  recently introduced by M.Aguiar and F.Ardila.
  We investigate the existence of a
  cointeracting bialgebra structure on EGP’s. We show that a suitable notion of
  cointeraction exists, not in the classical comodule sense,
  but via the framework of measuring algebras.
The  comodule-type map assigns to each polyhedron the sum of pairs of face
and tangent cone at the face.
EGP’s and affine-cone EGP’s form cointeracting bimonoids in species with
EGP as a third measuring structure.

EGP’s are in bijection to 
submodular functions. For an EGP, we also describe explicitly the
submodular functions of its faces and tangent cones.
The braid fan and its relation to
preorders play a key role in this description.
\end{abstract}

\section{Introduction}

Extended generalized permutahedra are polyhedra
which relate to and connect a range of significant structures:
posets and preorders, graphs and hypergraphs, matroids, and building sets.

In \cite{AA2023} M.Aguiar and F.Ardila established a Hopf monoid structure
on extended generalized permutahedra. There is a product $\mu$, the
cartesian product
of polyhedra, and a coproduct $\Delta$: For an extended generalized
permutahedron $\Pi$ on $\RR I$, where $I$ is a finite set, then for
 (a distinguished selection of) $S \sus I$ there is a face of
$\Pi$ which is a product $\Pi_{|S} \times \Pi_{/ S}$, where
the two factors lie in $\RR S$ and $\RR I \backslash S$ respectively.

Research in combinatorial Hopf algebras in the last fifteen years,
with \cite{CEM} as a founding article,
has found that many combinatorial Hopf algebras have
a cointeracting bialgebra. This bialgebra has the same product $\mu$, but
a different coproduct $\delta$, and $\mu, \Delta$ and $\delta$ relate,
they {\it cointeract}, in a specific way. This raises the questions:
Is there such a cointeraction structure $\delta$ for extended
generalized permutahedra (EGP's)? Furthermore, what features of
EGP's does $\delta$ reflect?

   We establish that there is indeed a cointeraction for EGP's,
but in a distinct technical sense than what has been found in the
research done in this direction.
The standard notion of cointeraction is the following: A bialgebra
(usually a Hopf algebra) $(A,\mu_A, \Delta_A, \eta_A, \epsilon_A)$
and a bialgebra $(B, \mu_B, \delta_B,\eta_B, \delta_B)$ cointeract if $A$ is a
(left) comodule over $B$:
\begin{equation} \label{eq:intro-delta} \delta^\prime : A \pil B \te A
\end{equation}
and all structure maps for $\mu_A, \Delta_A, \eta_A$ and $\epsilon_A$
are morphisms of $B$-comodules.

The situation we encounter is different. The notion in the literature
that best reflects our situation is {\it measuring}. It
originates in the situation of algebras and coalgebras, in particular
as {\it measuring coalgebras},  but has shown itself natural 
in great categorical generality, \cite{Hy-Measure, North}.
It is actually the dual notion, that of {\it measuring
  algebra} that we encounter. It turns out that
the comodule map $\delta^\prime$ of \eqref{eq:intro-delta} giving
(standard) cointeracting bialgebras is equivalent to the algebra $B$ measuring
the bialgebras $A$ and $A$. In our situation $A$ will
be extended generalized permutahedra and $B$ those EGP's which are
affine cones (translates of cones). We show there is a right
comodule map
\begin{equation} \label{eq:intro-delta2}
  \delta : A \pil A \te B
\end{equation}
(here $\te$ is a monoidal product), such that the algebra $A$ (in
$A \te B$) measures the bialgebras $A$ (to the left) and $B$
(in contrast to $B$ measuring the bialgebras $A$ and $A$).
In fact, this comodule map is meaningful in full generality. $A$
may be all polyhedra, and $B$ all polyhedra which are affine cones.
But the {\it cointeraction} (or {\it measuring}
by an algebra) of {\it bialgebras}
only makes sense in the setting above when $A$ and $B$ are restricted to EGP's.

The comodule map \eqref{eq:intro-delta2} sends a polyhedron $P$, to the sum
\[ \sum_{F \text{ face of } P} (F, \cone_F(P)) \]
where we sum over all faces $F$ of $P$, and $\cone_F(P)$ is the
affine cone of $F$ at $P$, in the literature also called the {\it tangent cone}
of $P$ at the face $F$. The affine lineality space of $\cone_F(P)$
is the affine linear span of $F$, and the faces of $\cone_F(P)$ are in
bijection to the faces $G$ of $P$ with $F \sus G \sus P$.

Now, specializing $\delta$ to EGP's on a space
$\RR I$ where $I$ is a finite set, these are in bijection to submodular
functions $z : \pow(I) \pil \hat{\RR} = \RR \cup \{ \infty \}$,
where $\pow(I)$ is the power set of $I$.
For the Hopf monoid structure \cite{AA2023} on EGP's, on submodular
functions the coproduct $\Delta$ becomes
\[ z \mapsto \sum_{\stackrel{S \sus I}{z(S) < \infty}} (z_{|S}, z_{/ S}).
  \]
  Here $z_{|S}$ is the restriction of $z$ to $\RR S$, and $z_{/ S}$
  the corestriction to $\RR I \backslash S$ (see Section \ref{sec:submod}).
  This raises the question: Which form does the comodule map $\delta$ in
  \eqref{eq:intro-delta2} take for submodular functions?
  This requires expressions for the submodular function of
  $F$ and of $\cone_F(P)$ for faces $F$ of $P$.

  To achieve this, we recall the fundamental relation of EGP's to the
  braid fan. This is the full fan of $\RR I$ defined by the hyperplane
  arrangements $x_i = x_j$ for distinct $i,j$ in $I$.
  A {\it braid cone} is a cone
  which is defined by an intersection of a selection of half-spaces
  $\{x_i \leq x_j \}$. Braid cones turn out to be in bijection
  with preorders $P$ on $I$, and we write $k(P)$ for the corresponding
  braid cone. 

  By definition, an extended generalized permutahedron,
  is a polyhedron whose normal fan is a coarsening of a subfan of the
  braid fan. The normal cone of a face $F$ of this polyhedron is a
  braid cone $k(P)$
  for some preorder $P$. In Theorem \ref{thm:egp-face} we give the
  submodular function of the face $F$.
 Denoted $z_P$, it is constructed from $z$ and $P$. 
 Theorem \ref{thm:egp-cone} gives the submodular function of
 the tangent cone $\cone_F(P)$ of $P$ at $F$.
 Denoted $z^P$, it is simply given by $z^P(S) = z(S)$ whenever $S$ is
 a down-set of $P$, and $z^P(S) = \infty$ otherwise.

 \medskip
 The organization of the article is as follows.

 \medskip
 \noindent{{\bf Part I. Cones and polyhedra.}}
 This parts contains general material on cones and polyhedra, before
 in Part II introducing extended generalized permutahedra.
 Section \ref{sec:cones} considers cones in a real vector space and
 for a given cone $C$ establishes a Galois connection between the subcones
 and supercones of $C$. Section \ref{sec:species} recalls the notion
 of species. Monoids, comonoids and bimonoids in species are the
 appropriate notions for formulating our results.
 Section \ref{sec:polyh} considers polyhedra in a real vector space
 and the defining equations of its faces and tangent cones. It also
 establishes the comodule map $\delta: \poly \pil \poly \times \conespa$
 for the species of polyhedra and affine cones.

\medskip
 \noindent{{\bf Part II. Extended generalized permutahedra and submodular
     functions.}}
   Section \ref{sec:braid} recalls the braid fan in $\RR I$ and how
 braid cones corresponds to preorders. A specialization of the Galois
 connection in Section \ref{sec:cones} gives a Galois connection
 between  the sub- and super-preorders of a fixed preorder $P$. We
 recall how these describe the faces of a braid cone.
 Section \ref{sec:submod} recalls the notions of submodular and of modular
 functions. Section \ref{sec:egp} give the bijection between submodular
 functions and extended generalized permutahedra (EGP), and the restricted
 bijection between modular functions and EGP's which are affine cones.
 Theorems \ref{thm:egp-face} and \ref{thm:egp-cone}
 describes the submodular functions of i) faces and ii) tangent cones of
 an EGP.
 Section \ref{sec:cointer} gives our main theorem on the cointeracting
 bimonoid species of EGP's and the EGP's which are affine cones.

 \medskip
 \noindent{{\bf Appendix. Measuring algebras and species}}
 We present the notions of measuring coalgebras and algebras. We
 show that the standard notion of cointeracting bialgebras can be
 reformulated with the notion of measuring algebra. We give the
 variant of two bialgebras being measured by an algebra
 which applies in our setting.

 \medskip
 \noindent{{\bf Note.} The present article is a substantial rewriting of
   a previous unpublished preprint ``Submodular functions,
   generalized permutahedra,
   conforming preorders, and cointeracting bialgebras''.
   The present article shifts focus to polyhedral geometry, and introduces
   the notion of measuring algebra for the cointeraction of bialgebras.

\section{Cones}

\label{sec:cones}

Let $V$ be a finite-dimensional vector space over the real numbers $\RR$
so $V$ is isomorphic to $\RR^n$ for some $n$. We do not at this point
fix a basis, so we use $V$ as notation for our vector space.

We consider {\it polyhedral cones} in $V$, cones defined by a finite
set of linear forms
\[ \{ x \in V \, | \, w_i(x) \leq 0, \, i \in I\}. \]
For a cone $C$, its negative cone is $-C = \{-c \, | \, c \in C \}$. 
The {\it lineality space} of $C$ is $C \cap -C$, the largest linear
subspace of $V$ contained in $C$.

\subsection{Faces and cofaces}
We shall consider (polyhedral) subcones $D \sus C$ and
supercones $C \sus E$, and
partially order these by inclusion. Denote these respectively by
$\subc(C)$ and $\supc(C)$. There are order-preserving maps
  \begin{align} \label{eq:cones-galois} \subc(C) & \bihom{\Phi}{\Gamma} \supc(C) \\
  D & \mapsto C + (-D)  \notag\\
  C \cap (-E) & \mapsfrom E \notag
  \end{align}

  \begin{lemma} This is a Galois connection.
  \end{lemma}

\begin{proof}
  We must show $\Phi(D) \sus E$ iff $D \sus \Gamma(E)$.
  So assume $\Phi(D) = C + (-D) \sus E$. Let $d \in D$. We have $-d \in E$
  and so $d \in -E$. Thus $d \in C \cap (-E)$.

  Conversely, let $D \sus \Gamma(E) = C \cap (-E)$. For $d \in D$ we have $-d \in E$,
  and so $C + (-D) \sus E$.
\end{proof}
    
As a consequence, by general properties of Galois correspondences,
the images $\im(\Phi)$ and $\im(\Gamma)$ are in bijection.

\begin{lemma} \label{lem:cones-face}
    The image $\im(\Gamma)$ are the faces of $C$.
  \end{lemma}

  \begin{proof}
    We first show that the elements of $\im(\Gamma)$ are faces. 
    The lineality space $ E \cap (-E)$ of $E$ is a face of $E$. Let
    $H$ be a supporting hyperplane for this face, so $E$ is contained
    in one of the half spaces defined by $H$, and $H \cap E = E \cap (-E)$.
    Consider $G(E) = C \cap (-E)$. We claim that $H \cap C = C \cap (-E)$:
    \begin{align*} H \cap C = H \cap (E \cap C) = 
      (H \cap E) \cap C  = & (E \cap (-E)) \cap C \\
  = & (E \cap C) \cap (-E) =  C \cap (-E).
  \end{align*}
  Now, for a face $F$ of $C$, we show it is in the image of $\Gamma$.
  We can write $F = C \cap H$ for some hyperplane, where $C$ is contained
  in the half plane $H^+$, and so $H^+$ is a supercone of $C$. Then
  \[ F = C \cap H = C \cap (H^+ \cap -H^+) = (C \cap H^+) \cap (-H^+)
    = C \cap (-H^+). \]
\end{proof}

  The elements of $\im(F)$ are the {\it cofaces of $C$}.
  We thus have a bijection
  \[ \face(C) \overset{1-1}{\longleftrightarrow} \coface(C). \]

  To the face $D$ corresponds the coface $C + (-D)$. We also
  denote this as $\cone_D(C)$, the {\it cone of $C$ at the face $D$}.
  Note that the minimal face of this, the lineality space of $\cone_D(C)$, is
  the linear span $D + (-D)$ of $D$. 

  \begin{lemma} \label{lem:cone-FG}
    For a face $F$, the faces of $\cone_F(C)$ are the $\cone_F(G)$
    as  $G$ runs over the faces of $C$ with $F \sus G \sus C$.
  \end{lemma}

  \begin{proof}
    Let $F \sus G \sus C$ be a face of $C$ containing the face $F$.
    Then $G = C \cap (-E)$ for some supercone $E \supseteq C$, and
    so $-G \sus E$.
    Then $E \supseteq C + (-G) \supseteq C + (-F)$, so $E$ is also
    a supercone of $C + (-F)$. 
    We claim that $(C + (-F)) \cap (-E)$ is $G + (-F)$, and so the latter
    is a face of $\cone_F(C) = C + (-F)$.  Clearly $G + (-F)$ is
    contained in this intersection.
    If $c + (-f)$ is in $(C + (-F)) \cap (-E)$
    then $c$ is in $-E$, and so in $G = C \cap (-E)$. Thus
    $c + (-f)$ is in $G + (-F)$.

Conversely, let $G^\prime$ be a face of $C + (-F)$, so
$G^\prime = (C + (-F)) \cap (-E)$ for some supercone $E \supseteq C + (-F)$.
Let $G = C \cap (-E)$. We claim that $G^\prime = G + (-F)$.
Clearly $G^\prime \supseteq G + (-F)$.  Let then $c + (-f)$ be in
$G^\prime = (C + (-F)) \cap (-E)$.
Then $c \in -E$ and so $c$ is in $G = C \cap (-E)$. 
  \end{proof}

\subsection{Directions and duals} \label{subsec:duals}
An element in the dual vector space $W = V^*$ is a {\it direction}.
If $w$ assumes a maximum value on a cone $C$, this value must be zero.
The {\it dual cone} $C^\vee$ of $C$ is the set of all $w \in W$ which assume
this maximal value on $C$. The elements of $C^\vee$ are the {\it bounded
  directions} for $C$. If $C$ is defined by the hyperplanes
$\{ v \in V \, |\, w_i(v) \leq 0, \, i \in I \}$,
the dual cone $C^\vee$ is generated by the rays
$\{ w_i \in W \, | \, i \in I \}$.

Duality is an involution and it reverses inclusions:
if $D \sus C$, then $C^\vee \sus D^\vee$. 
Furthermore, sums goes to intersection and vice versa:
\[ (C+D)^\vee = C^\vee \cap D^\vee, \quad (C \cap D)^\vee = C^\vee + D^\vee. \]
In particular the Galois connection \eqref{eq:cones-galois} becomes
the dual Galois connection:
\begin{align} \label{eq:cones-galoisdual}
  \supc(C^\vee) & \bihom{}{} \subc(C^\vee) \\
  E^\prime  & \mapsto C^\vee \cap  (-E^\prime)  \notag\\
  C^\vee +  (-D^\prime) & \mapsfrom D^\prime. \notag
\end{align}

For a bounded direction, the set
\[ C_w = \{ x \in V \, | \, w(x) = 0 \} \]
is a face of $C$. All faces are obtained this way. Such $w$ give
the supporting hyperplanes of a face.

\begin{lemma} \label{lem:cone-w} Let $F$ be a face of $C$.
\begin{itemize}
\item[a.] $w$ is a bounded direction for $C$ with $F \sus C_w$ if and only if
$w$ is a bounded direction for $\cone_F(C)$.
\item[b.] In this case $\cone_F(C_w) = \cone_F(C)_w$.
\end{itemize}
\end{lemma}

\begin{proof}
  a. Suppose $w(F) = 0$, and $w(C) \leq 0$. Then $w$ also has maximum value
  $0$ on $C-F$. For the converse, if $w$ has maximum value $0$ on $C-F$,
  then we cannot have $w(f) \neq 0$ for any $f \in F$. So $w(F) = 0$.

  b. We have $w(F) = 0$. Then
  \begin{align*} \cone_F(C_w) & = C_w - F = \{ c-f \, | \, w(c) = 0 \} \\
    \cone_F(C)_w & = \{ c-f \, | \, w(c-f) = 0\},
  \end{align*}
and these two sets are equal.
\end{proof}

 \section{Species}

\label{sec:species}

We recall briefly the notion of species, monoids, comonoids and bimonoids
in species. We also show how polyhedral cones become a bimonoid in species.

Let $\setx$ be the category of sets with bijections, and
$\cC$ a category. 
A {\it species} in $\cC$ is a functor
\[ \Sp : \setx \pil \cC. \]
This notion was introduced by A.Joyal \cite{J1981}, see
\cite{AM2010, BLL1998} for books developing this notion in detail.


\subsection{Monoids and comonoids in species}
Assume now $(\cC, \te, \idmon)$ is a symmetric monoidal category
with product $\te$ and identity element $\idmon$. 
We also assume it has 
finite products and coproducts and these are equal, and we use
$\coprod$ to denote these. In particular $\cC$ has a zero object $\bfnull$.
The typical example of $\cC$ is
the category of vector spaces with tensor product, although
our setting for $\cC$ will be slightly different.

Two monoidal products may then be defined on the category of species.
For species $\Sp$ and $\Tp$:

\begin{itemize}
\item The Cauchy product:
\[ (\Sp \tec \Tp)[I] = \bigoplus_{I = S \sqcup T}
  \Sp[S] \te \Tp[T], \quad \text{identity object }
\benc[I] = \begin{cases} \idmon, & I = \emptyset \\ \bfnull, & I \neq \emptyset.
\end{cases} \]
\item The Hadamard product:
  \[ (\Sp \teh \Tp)[I] = \Sp[I] \te \Tp[I], \quad \text{identity object }
\benh[I] = \idmon, \text{ for every } I. \]
\end{itemize}
With each of these monoidal products we have the notions of
{\it monoids, comonoids}, and {\it bimonoids}.

\medskip
\subsection{The category $\setn$} \label{subsec:species-setn}
We will not work with vector spaces as our symmetric monoidal category,
since our objects are most simply and naturally considered as sets.
A set can be made into 
the basis of a vector space, but that would be introducing unnecessary
structure.

Instead the symmetric monoidal category which will be convenient for us
is the following, introduced in \cite{F-CHA}. 
For a set $Y$, let $\Homf(Y, \NN)$ be the set of maps $p : Y \pil \NN$
such that $p(y)>0$ for only a finite set of $y$'s in $Y$.
So $\Homf(Y,\NN)$ identifies as finite multisubsets of $Y$.

Let $\setn$ be the category whose objects are sets
and whose morphisms $f : X \pil Y$ are maps $X \pil \Homf(Y,\NN)$.
So an element $x \in X$ maps to a finite sum $\sum_y n^x_y y$.
The category $\setn$ has a $0$-object, the empty set, and finite products
and coproducts are equal
\[ X \prod Y = X \coprod Y = (X \times{1}) \cup (Y \times {2}). \]
Our monoidal product on $\setn$ will be the (classical cartesian)
product $X \times Y$ of sets.
\subsection{A variation of species}

Let $\vectx$ be the category of finite-dimensional vector spaces,
with invertible linear maps. Species are normally defined as functors
from the category of sets with bijections as morphisms.
We shall consider the following variant of species, functors
\[ \Pp : \vectx \pil \cC, \]
with $\cC$ as before, a symmetric monoidal category with finite products
and coproducts equal. Let us call this a {\it v-species}.

For this setting, the Hadamard product of two species is well defined.
The Cauchy product is not meaningful, but we can make the
following ad hoc definition of a monoid in v-species.
This is a v-species $\Pp$ together with product maps  and unit map
\[ \Pp(V) \otimes \Pp(W) \pil \Pp(V \oplus W), \quad \benc \pil \Pp, \]
fulfilling the natural axioms of associativity and unitality for
a monoid. 

\subsection{Bimonoids of species of cones}

In the real vector space $V$, denote by $\conesp(V)$ the set of
polyhedral cones in $V$. Then $\conesp$ is a v-species in $\setn$.
It is a monoid in species by taking products of cones:
\begin{align*} \mu: \conesp(V) \times \conesp(W) & \pil \conesp(V \oplus W) \\
 (C_1, C_2) & \mapsto C_1 \times C_2.
\end{align*}
It is moreover a comonoid in v-species for the Hadamard product, with
coproduct
\begin{align} \label{eq:cone-delta}
  \delta_{\mathsf c}: \conesp(V) & \pil \conesp(V) \times \conesp(V) \\ \notag
  C & \mapsto \sum_{F \text{ face of } C} (F, \conesp_F(C)).
\end{align}
Coassociativity follows from Lemma \ref{lem:cone-FG} and that for
faces $F \sus G$, then
\begin{align*} \cone_{\cone_F(G)}(\cone_F(C)) = & \cone_F(C) - \cone_F(G) \\
  = & (C-F) - (G-F) = C-G = \cone_G(C). \end{align*}
The counit is:
\[ \epsilon_{\mathsf c} : \conesp \pil \benh, \quad C \mapsto \begin{cases} (* \mapsto 1), & C \text{ a linear
      space} \\ (* \mapsto 0), & \text{otherwise.} \end{cases}
\]

Moreover, faces of
$C_1 \times C_2$ are of the form $F_1 \times F_2$ with $F_1$ a face of $C_1$
and $F_2$ a face of $C_2$, and
\begin{align*} \cone_{F_1 \times F_2}(C_1 \times C_2)
  & = C_1 \times C_2 - F_1 \times F_2 \\
&  = (C_1 - F_1)\times (C_2 - F_2) = \cone_{F_1}(C_1) \times \cone_{F_2}(C_2).  
\end{align*}
Thus $\delta_{\mathsf c}( C_1 \times C_2) = \delta_{\mathsf c}(C_1)
\times \delta_{\mathsf c}(C_2)$,
and so this makes $\conesp$ into a {\it bimonoid of v-species}.

\section{Polyhedra}

\label{sec:polyh}

We recall basic notions for polyhedra. We are particularly concerned
with the tangent cone of a face of the polyhedra. We define a bimonoid
of polyhedra and the tangent cone is a main ingredient for this. 
Again $V$ be a finite dimensional real vector space, and $W = V^*$ its
dual space.

\subsection{Polyhedra}
\begin{definition}
A polyhedron $P$ in $V$ is a subset of $V$
defined by a finite set of forms $w_1, \ldots, w_n
\in W = V^*$, and real numbers $b_1, \ldots, b_n$, as
\begin{equation} \label{eq:polyh-wb}
  P = \{ x \in V \, |\, w_i(x) \leq b_i, \, i  \in I\}.
  \end{equation}
$P$ is a {\it cone} if every $b_i = 0$, and an {\it affine cone}
if it is a translation of a cone.
\end{definition}

Denote the following:
\begin{itemize}

\item $\poly(V)$: the set of polyhedra in $V$,
\item $\conespa(V)$: the subset of affine cones in $V$,
\item $\conesp(V)$: the subset of cones of $V$,
\end{itemize}

The following is straight-forward.
\begin{lemma} \label{lem:polyh-conea}
  The inequalities in \eqref{eq:polyh-wb} define an affine
  cone iff the equalities $w_i(x) = b_i$ has a solution. The solution
  set is the affine lineality space of the cone.
\end{lemma}

By forcing equalities in \eqref{eq:polyh-wb} for a subset $J \sus I$,
one gets a face of $P$, and any face is obtained this way.

Let $u$ be any point of $P$. The {\it recession cone} of $P$ is 
\[ \res(P) = \{ x \in V \, | \, u + \lambda x \in P \text{ for all }
  \lambda > 0. \}
\]
By \cite[Prop.1.12a]{Zi}, it is the cone defined by the inequalities
$w_i(x) \leq 0$, where $w_i$ are the hyperplanes in \eqref{eq:polyh-wb}
defining $P$.

\subsection{The normal fan}
Given polyhedron $P$, and a direction $w \in W$. We say $w$ is bounded
if it assumes a maximum value $c$ on $P$.
The set
\[ P_w = \{ x \in P \, | \, w(x) = c \} \]
is then a face of $P$. Every face may also be obtained this way.
For a face $F$, the set of all directions $w$
attaining their maximum value on $F$  is the {\it normal cone} $c(F)$ of $F$.
If $F \sus G$ are faces, the cone $c(G)$ is a face of the cone $c(F)$. 
The set of normal cones $c(F)$ as $F$ ranges over the faces of $P$, form
a fan, the {\it normal fan} $N_P$.
The union of all these normal cones form a cone which is dual
to the recession cone of $P$, by \cite{Zi}, Theorem 1.2 and Proposition 1.12ii.
In particular if $P$ is a polytope, the normal fan is full: The union of
its cones is the full space $V$. 

The normal cone $c(P)$ is the
dual of the linear span of $P$, and so is a linear space, which is
the lineality space of each cone $c(F)$. If $P$ is also full-dimensional,
all cones $c(F)$ are pointed.
The maximal cones of the normal fan $N_P$, are dual to the minimal faces of $P$.
All minimal faces of $P$ have the same dimension $e$,
and their dual give the maximal cones in the normal fan $N_P$, of
dimension $\dim_k V - e$. 

We remark the following: Given a fan $N$ in $V$, which is a subdivision
of a cone, there may or may not exist a polyhedron with normal fan $N$. 

\subsection{Tangent cones at faces}

\begin{definition} \label{def:poly-cone}
  For a face $F$ of $P$, let $u$ be an interior point of $F$.
  {\it The tangent cone of $P$ at $F$} is the affine cone
\[ \cone_F(P) = \{ u + x \in V \, | \, u + \lambda x \in P \text{ for all }
  \lambda \in (0,\epsilon) \text{ for some } \epsilon > 0 \}.
\]
\end{definition}

\begin{proposition} Let $P$ be defined by $\{ w_i \leq b_i \, | \, i \in I\}$.
  Let the (non-empty) face $F$ be defined by the additional equalities
  $\{w_i = b_i \, | \, i \in J \}$ for some (largest possible) subset $J \sus I$
  (the so-called active set for $F$). 
  Then $\cone_F(P)$ is defined by $\{ w_i \leq b_i \, | \, i \in J \}$.
\end{proposition}

\begin{proof}
  We translate this so the origin becomes an interior point of the
  face $F$. Then $F$ is defined
  by $w_i = 0$ for $i \in J$ together with $w_i \leq b_i^\prime$ for $i \in I$
  (and $b_i^\prime$ obtained from $b_i$).
  We must
  have each $b_i^\prime \geq 0$, as they must be satisfied by the origin. If
  $x$ fulfills $w_i(x) \leq  0$ for $i \in J$, then some small
  positive multiple $\lambda x$ will fulfill
  all $w_i \leq b_i^\prime$ for $i \in I$,
  and so is in $\cone_F(P)$. Conversely any point $x$ in $\cone_F(P)$
  must fulfill $w_i(\lambda x) \leq 0$ for $i \in J$ and $\lambda$ small
  and positive,  and so fulfills
  $w_i(x) \leq 0$ for $i \in J$.
\end{proof}

\begin{corollary} \label{cor:polyh-face}
  Faces of $\cone_F(P)$ are in bijection to faces $F \sus G \sus P$.
  \begin{itemize}
  \item[a.] The associated face in $\cone_F(P)$ is $\cone_F(G)$,
  \item[b.] $\cone_{\cone_F(G)}(\cone_F(P)) = \cone_G(P)$.
  \end{itemize}
\end{corollary}

\begin{proof}
  Let $F$ be defined by $\{ w_i = b_i \, | \, i \in J \}$ where $J \sus I$
  (the active set), together with $w_i \leq b_i$ for $i \in I$.
  For short we say $F$ is
  {\it defined by the pair} $(J,I)$. Let $(K,I)$, where $K \sus J$
  define the face $G$ of $P$. Then $G$ is defined by the inequalities
  \[ \{ w_i \leq b_i \, | \, i \in I \} \cup
    \{ -w_i \leq -b_i \, | \, i \in K \}. \]
  Then $\cone_F(G)$ is defined by
  \[ \{w_i \leq b_i \, | \, i \in J \} \cup \{ -w_i \leq -b_i \, |\, i \in K\},
\]
which is
\[ \{ w_i = b_i \, |\, i \in K \} \cup \{w_i \leq b_i \, | \, i \in J \}, \]
and so is a face of $\cone_F(P)$.

Conversely, let $G^\prime$ be a face of $\cone_F(P)$ defined by the
pair $(K,J)$.
The pair $(K,I)$ induces a face $G$ of $P$, and $G$ contains $F$.
We claim that this pair defines $G$. 

If there was an
$i_0 \in I \backslash K$ with $w_{i_0} = b_{i_0}$ on $G$,
then $F$ would satisfy this
and so $i_0 \in J$. Translating $F$ to contain the origin, we would get
$b_{i_0} = 0$.
Let $u+x$ be a point in $G^\prime$, where $u$ is interior to $F$.
Then $w_j(u) = 0$ for $j \in J$ and $w_j(u+x) = 0$ for $j \in K$.
Thus for $\lambda$ small, positive, $u + \lambda x \in P$, and
$w_j(u + \lambda x) = 0$ for $j \in K$. 
Then $u + \lambda x \in G$, and so
$w_{i_0}(u + \lambda x) = 0$, for all these $\lambda$. Then $w_{i_0}(x)=0$,
and so $w_{i_0}(u+x) = 0$ for points in $G^\prime$, a contradiction.
Thus the pair $(K,I)$ defines the face $G$. Part b. follows, since
both cones are defined by
 \[ \{ w_j \leq b_j \, | \, j \in K \} \]
\end{proof}

\begin{corollary} \label{cor:poly-w}
  The direction $w$ is bounded for $P$ with the face $F \sus P_w$
  iff $w$ is a a bounded direction for $\cone_F(P)$. In this case:
\[ \cone_F(P_w) = \cone_F(P)_w. \]
\end{corollary}

\begin{proof} Let $F$ be defined by the pair $(J,I)$. 
  We may translate so the origin becomes interior in $F$.
  Also, the inequalities $w_i \leq b_i$
  defining $P$ will have every $b_i \geq 0$.

  Suppose $w$ bounded on $P$ with $F \sus P_w$. The $w$ takes
  maximum value $0$ on $P$.
  If $w$ is not bounded on $\cone_F(P)$, so $w(x)$ became arbitrary large,
  then we must have $w(x) > 0$ for some $x$ fulfilling $w_i(x) \leq 0$
  for $i \in J$. 
  But $\lambda x $ would be in $P$ for $\lambda$ small positive, and so $w$
  could not have maximal value $0$ on $P$. Thus $w$ is bounded
  on $\cone_F(P)$.
  Now the defining set for $\cone_F(P_w)$ and $\cone_F(P)_w$ are
  seen both to be
  \[ \{w = 0 \} \cup \{ w_i \leq b_i, \, i \in J \}. \]

  If $w$ bounded on $P$ but $F$ not contained in $P_w$,  then as $\cone_F(P)$
  contains the lineality space of $F$, $w$ is either not bounded on
  there, or it is constant there, say with value $c$ (which
  is less than the maximum of $w$ on $P$).
  In this latter case, there must be
  $u+ \lambda x$ 
  in $P$, with $u$ in the interior of $F$,
  with  $w(u + \lambda x) > c$. But then $w$ is unbounded on $\cone_F(P)$.
  
  If $w$ is not bounded on $P$, we have
  $w(y)$ arbitrarily large for $y = u+x \in P$, and so also $w(u + x)$
  arbitrarily large for $u$ in the interior of $F$.
\end{proof}

\subsection{Bimonoids of polyhedra}

The v-species $\poly(V)$ becomes a monoid in v-species by the
natural product of polyhedra:
\[ \mu : \poly(V) \times \poly(W) \pil \poly(V \oplus W). \]

We can also make $\poly(V)$ a right comodule monoid over
the affine cones species $\conespa(V)$: First there is a map
\begin{align} \label{eq:poly-delta}
  \delta_{\mathsf p} : \poly(V) & \longrightarrow \poly(V) \times \conespa(V) \\
  P & \mapsto \sum_{F \text{ face of } P} (F,\cone_F(P)). \notag
\end{align}
On $\conespa(V) \sus \poly(V)$, the restriction of $\delta_{\mathsf p}$,
denoted $\delta_{\mathsf c}$, is a general version of \eqref{eq:cone-delta},
making $\conespa$ a bimonoid in species.

\begin{proposition} The map $\delta_{\mathsf p}$ makes $\poly$ a comodule
  in v-species over the comonoid $\conespa$.
\end{proposition}

\begin{proof}
  The composition $(\delta_{\mathsf p} \times 1) \circ \delta_{\mathsf p}$ sends
  \[ P \mapsto \sum_{G \sus P} (G, \cone_G(P))
  \mapsto \sum_{F \sus G \sus P} (F, \cone_F(G), \cone_G(P)).\]
  The composition $(1 \times \delta_{\sfc}) \circ \delta_{\mathsf p}$ sends
  \[ P \mapsto \sum_{F \sus P, \, G^\prime \sus \cone_F(P)}
      (F, G^\prime,  \cone_{G^\prime}(P)). \] 
    By Corollary \ref{cor:polyh-face},
    faces $G^\prime$ of $\cone_F(P)$ are in bijection to faces
    $F \sus G \sus P$ and $G^\prime = \cone_F(G)$.
    Thus these two maps are equal.
    
    Lastly, note that
    $(1 \times \epsilon_{\mathsf c}) \circ \delta_{\mathsf p} = 1_{\poly}$. 
  \end{proof}

  Together, the product $\mu$ and the coproduct $\delta_{\sfp}$ make
  $\poly(V)$ into a {\it bimonoid in v-species} due to:
\begin{align*} \cone_{F_1 \times F_2}(P_1 \times P_2) = 
 \cone_{F_1}(P_1) \times \cone_{F_2}(P_2).  
\end{align*}

Extending the counit $\epsilon_{\mathsf c}$ on $\conesp$ and $\conespa$,
we get a morphism
of species
\[ \epsilon_{\mathsf p} : \poly \pil \benh, \quad P \mapsto
  \begin{cases} (* \mapsto 1), &
P \text{ affine linear space} \\ (* \mapsto 0), & \text{otherwise.} \end{cases}
\]
We note that $(\epsilon_{\mathsf p} \times 1) \circ \delta_{\mathsf p}$ sends
\[ P \mapsto \sum_{\text{vertex } v \text{ of } P} \cone_v(P), \]
the sum over the {\it set of cones associated} to $P$, \cite[Ex.9, p.49]{Grun}.

\section{Preorders and the braid fan}

\label{sec:braid}
The hyperplanes $x_i = x_j$ define a central hyperplane
arrangement, and an associated fan in $\RR^I$, the {\it braid fan}.
Polyhedra
whose normal fan is a coarsening of a subfan of the braid fan,
are by definitions the extended generalized permutahedra (EGP).
Braid cones, defined as intersections of half spaces $x_i \leq x_j$
for a selection of pairs $(i,j)$, are precisely in bijection to
preorders on $I$. This gives an intimate relation between EGP's,
braid cones and preorders on $I$. 

\subsection{Preorders}
A preorder $P = (I, \leq)$ on a set $I$ is a relation which is:
\begin{itemize}
\item[] {\it Reflexive:} $x \leq x$ for $x \in I$,
\item[] {\it Transitive:} $x \leq y$ and $y \leq z$ implies $x \leq z$.
\end{itemize}
The preorder is a {\it partial order} if it is also:
\begin{itemize}
\item[] {\it Anti-symmetric:} $x \leq y$ and $y \leq x$ implies $x = y$.
\end{itemize}

A subset $D \sus I$ is a {\it down-set} for $P$ if whenever $y \in D$ and
$x \leq y$, then $x \in D$. If $x \in I$ we write $\da{x}$
for the down-set $\{ y \in I \, | \, y \leq x \}$.
We have a corresponding notion of {\it up-set} and $\ua{x}$.
We write $\leq_P$ for $\leq$ when we need to be explicit about
which preorder we refer to.
Given a preorder $P = (I, \leq)$, we have an {\it opposite preorder}
$P^{\op} = (I,\leq^\op)$ given by $x \leq^\op y$ if $y \leq x$. 

Given $P = (I, \leq)$, two $x,y \in I$ are {\it comparable} for $P$
if $x \leq y$ or $y \leq x$. If only one of them holds, say $x \leq y$,
then $y$ is {\it strictly} bigger than $x$.
A preorder is {\it total} if any two elements are comparable. It is
{\it linear} if for any two elements one of them
is strictly bigger than the other.
A preorder $P$ is an {\it equivalence relation} if $a \leq b$ implies
$b \leq a$.

\medskip
The set of preorders on $I$ is denoted $\Pre(I)$ and is itself partially
ordered by $P \preceq Q$ if $x \leq_P y$ implies $x \leq_Q y$ for
every $x,y \in I$.
In fact this is a lattice. 
For a given preorder $P$, we denote by $P^\circ$ the meet  $P \wedge P^{\op}$.
and by $P^\disc$ the join $P \vee P^{\op}$. These are both
equivalence relations. The classes of $P^\circ$ will be called
the {\it bubbles} of $P$, and ${\mathbf b}(P)$ will denote the
set of bubbles.
The classes of $P^\disc$ are the underlying
sets of the connected { components} of $P$, and ${\mathbf c}(P)$ will
denote this set of classes.

\begin{definition} \label{def:pre-lin}
  Let $P$ be a preorder and $L$ a total preorder
  Then $L$ is
  a {\it linear extension} of $P$ if:
  \begin{itemize}
  \item $P \preceq L$,
  \item The sets of bubbles $\mathbf b(P)$ and $\mathbf b(L)$ coincide.
  \end{itemize}
\end{definition}

\subsection{The braid fan} Let $I$ be a {\it finite} set.
The vector space $\RR I$ has basis $I$, and we write $e_i$ for $i$.
Elements in $\RR I$ are then $\sum_{i \in I} a_i e_i$ where $a_i \in \RR$.
Denote by $\RR^I = (\RR I)^*$ the dual space of linear maps $\RR I \pil \RR$.
The elements $y$ of $\RR^I$ are the {\it directions} for $\RR I$ and
identify as set maps $y : I \pil \RR$. 
We write $\{ \one_i \, | \, i \in I \}$ for the dual basis of
$\{e_i \, | \, i \in I \}$. Points in $\RR^I$ may be written
$b = \sum_{i \in I} b_i \one_i$ where $b_i \in \RR$.
For $S \sus I$, we write $\one_S = \sum_{i \in S} \one_i$.

\medskip
The {\it braid arrangement} on $\RR^I$ is defined by the hyperplanes $x_i
= x_j$, for $i, j$ distinct in $I$. It induces a fan on $\RR^I$, the
{\it braid fan} $\bfan_I$, see \cite[Sec. 1.3.5]{AA2023} or \cite[Sec. 10.2]
{AM2010}.
The maximal face cones of this fan,
called {\it chambers}, correspond to linear orders on $I$. If
$i_1 < i_2 < \cdots < i_n$ is such a total order,
the corresponding chamber is the set
of points $y$ in $\RR^I$ such that
\[ y_{i_1} \geq y_{i_2} \geq \cdots \geq y_{i_n}. \]
In general, the face cones of the braid fan are given by {\it total preorders}
on $I$. A total preorder $L$ on $I$ gives a face cone $k(L)$ consisting
of points $y$ in $\RR^I$ such that $y_i \geq y_j$ if $i \leq_L j$.
In particular, for each bubble $C = \{\ell_1, \ell_2, \cdots, \ell_r\}$
  of $L$ we have
\[y_{\ell_1} = y_{\ell_2} = \cdots = y_{\ell_r}. \]
(In the literature a total preorder on $I$ is often called a
set composition of $I$:
  It is an ordered sequence $(C_1, C_2, \ldots, C_r)$ of disjoint
  non-empty subsets of $I$, whose union is $I$, \cite[Sec.10.1.2]{AM2010}.)

  \medskip
  Each hyperplane $x_i = x_j$ determines two half spaces $x_i \geq x_j$
  and $x_j \geq x_i$. Following \cite[Sec.10.2.4]{AM2010} and
  \cite[Sec.3.4]{PRW2009}, a {\it braid cone} is an
  intersection of half spaces $x_i \leq x_j$ for a selection of
  pairs $(i,j)$ in $I \times I$. 
  A preorder $P$ on $I$ induces such a cone $k(P)$ in $\RR^I$, the
  intersection of half-spaces $x_i \geq x_j$ if $i \leq_P j$.
  Compactly, this may be described as $k(P) = \Hom(P^{\op}, \RR)$
where $\Hom$ means order-preserving maps. 
  An observation which we will use is that $1_A$ is in the
  braid cone $k(P)$ iff $A$ is a down-set in $P$.

  Proposition 3.5 of \cite{PRW2009} comprehensively describes cones
  in the braid arrangement and their bijection to preorders.
  We state an essential part of it below, and also add an equivalent
  description of braid cones. 
The following is Corollary 13.8 (and its proof)  in \cite{AM2010}. 
  
  \begin{proposition}
    \hskip 1mm {} \label{pro:braid-cone}

    \begin{itemize}
\item[a.] The correspondence
    \[ \Pre(I) \longrightarrow \text{ braid cones in } \RR^I,
        \quad P \mapsto k(P) \]
      is a bijection.
\item[b.] The maximal face cones
  of $k(P)$ are the $k(L)$ as $L$ ranges over the linear extensions
  of $P$.
  \end{itemize}
\end{proposition}

\begin{proposition} \label{pro:braid-kp}
  \hskip 1mm
  \begin{itemize}
  \item The linear span of $k(P)$ is $k(P^\circ)$. Thus
    its dimension is $|\mathbf b(P)|$.
  \item The lineality space of $k(P)$ is $k(P^\disc)$. Thus its
    lineality dimension is $|\mathbf c(P)|$.
  \end{itemize}
\end{proposition}

\begin{proof}
a.  Choose a general point $y$ in $k(P) = \Hom(P^{\op}, \RR)$.
  Then $y_p > y_q$ if $p < q$, while $y_p = y_q$ if $p,q$ are in
  the same bubble. We may freely perturb every coordinate in distinct
  bubbles and still be in $k(P)$. Thus the dimension is $|\mathbf b(P)|$.
  Since $P^{\circ} \preceq P$, the cones $k(P) \sus k(P^{\circ})$.
  But the latter is a linear space of dimension $|\mathbf b(P)|$.
  Since the former is a cone of this dimension, the latter must be
  the linear span of the former.

  b. If $p < q$ and $y$ is a point in the lineality space of $k(P)$, then
  both $y$ and $-y$ is in $k(P)$, and so we must have $y_p = y_q$.
  Thus $y$ must take an arbitrary constant value on each component, but
  we are free to choose distinct values on each component.
\end{proof}

\begin{corollary} 
  The braid cone $k(P)$ is a linear space iff $P$ is
  an equivalence relation. 
\end{corollary} 

\subsection{Extended generalized permutahedra}

All face cones of the braid fan contain the line $\RR \fone$.
It is therefore equivalent and convenient to consider the {\it reduced} braid
fan in the quotient space $\RR^I /\RR \fone$.
All face cones of the
reduced braid fan then become {\it pointed}, i.e. their minimal face is the
origin.

\medskip

\ignore{Given any polyhedron $\Pi$ in $\RR I$,
the points of $\Pi$ for which a direction in $\RR^I$ reaches maximal value,
form a face of $\Pi$, the maximal face for the direction.
We consider two directions
equivalent if they have the same maximal face $F$ of $\Pi$. Such
an equivalence class of directions
form an open cone.
The closure of this is the cone $k(F)$ associated to the face $F$.
These closed cones form a fan, the {\it normal fan} 
of the polyhedron. Note that smaller faces $F$ correspond to larger cones
$k(F)$.}


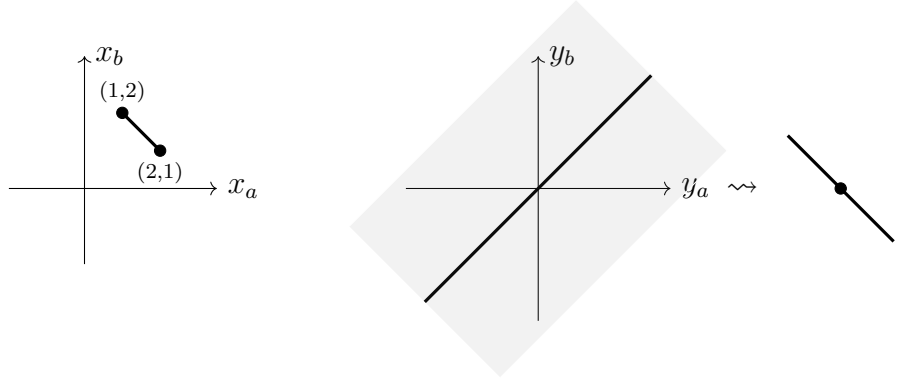
\begin{figure}

{\begin{tikzpicture}[scale=0.5, vertices/.style={draw, fill=black, circle,
inner sep=1.5pt}]

\draw [help lines, white] (-2,-1) grid (3,2);
        \draw[black, ->] (-2-12,0) -- (3.5-12,0) node [anchor=west] {$x_a$};
        \draw[black, ->] (0-12,-2) -- (0-12,3.5) node [anchor=west] {$x_b$};

\coordinate (a) at (2-12,1);
\coordinate (b) at (1-12,2);

\node[vertices] at (a) {}; 
\node[vertices] at (b) {};
\node[anchor=north] at (a) {$\scriptstyle{(2,1)}$}; 
\node[anchor=south] at (b) {$\scriptstyle{(1,2)}$};

\draw[black, very thick] (a)--(b);






\draw [help lines, white] (-2,-1) grid (3,2);
        \draw[black, ->] (-3.5,0) -- (3.5,0) node [anchor=west] {$y_a$};
        \draw[black, ->] (0,-3.5) -- (0,3.5) node [anchor=west] {$y_b$};

\coordinate (a) at (-3,-3);
\coordinate (b) at (3,3);
\coordinate (c) at (-5,-1);
\coordinate (d) at (1,5);
\coordinate (e) at (-1,-5);
\coordinate (f) at (5,1);

\draw[black, very thick] (a)--(b);

\draw[white, fill=gray, fill opacity = 0.1] (c)--(d)--(f)--(e) ;


\coordinate (p) at (-1.4 + 6.8,0);
\node at (p) {$\rightsquigarrow$}; 

\coordinate (x) at (-1.4 + 8,1.4);
\coordinate (y) at (1.4+8,-1.4);
\coordinate (z) at (0+8,0);

\node[vertices] at (z) {}; 
\draw[black, very thick] (x)--(y);
\end{tikzpicture}
}
 \caption{One-dimensional permutahedron and its fan and reduced fan }
        \label{fig:fanperm-1}
\end{figure}

Given fans $U$ and $V$ of $\RR^I$, the fan $U$ is a {\it coarsening} of $V$, if
\begin{itemize}
  \item Each cone of $V$ is contained in a cone of $U$
 \item  Each cone of $U$ is a union of cones of $V$.
 \end{itemize}

We recall Definitions 1.3.7 and 1.3.8 from \cite{AA2023}.

 \begin{definition}
 A {\it generalized permutahedron} (GP) is a  polytope whose
 normal fan is a coarsening of the braid fan. 
 An {\it extended generalized permutahedron} (EGP) is a polyhedron
 whose normal fan is a coarsening of a subfan of a braid cone.
\end{definition}

 (Note however that not every coarsening of the braid fan may
 be the normal fan of a generalized permutahedron
 \cite[Section 3]{MW2009}.)

 \subsection{Braid cones}
 \label{subsec:braidcone}
 
 The normal cone $K$ of a face of an extended generalized permutahedron
 it thus a cone that may be written as a union of cones of the braid
 arrangement. The following is stated at the beginning of
 Section 3.2 in \cite{PRW2009}, but an argument is not given. We
 provide one here, since it is an essential fact for us.

 \begin{proposition} \label{pro:braid-K}
   A cone $K$ is a union of cones
   from the braid arrangement if and only if $K = k(P)$ for some
   preorder $P$.
\end{proposition}
   
   \begin{proof}
     $k(P)$ are the order-preserving maps $P^{\op} \pil \RR$.
As $\RR$ is totally ordered, any element here
     is then in $L^{\op} \pil \RR$ for some linear extension $L$ of $P$.
     Thus $k(P)$ is the union of $k(L)$ as $L$ ranges over the linear
     extensions of $P$.
     
     Conversely, let $K$ be a cone which is a union of cones from the braid
     arrangements. Let $P$ consist of all pairs $(i,j)$ such that $y_i
     \leq y_j$ for every $y \in K$. This is seen to be a preorder
     with $K \sus k(P)$. We claim that $K = k(P)$.
     
Define an equivalence
  relation on $I$ by $i \sim j$ if $y_i = y_j$ for every $y \in K$.
  The equivalence classes are the bubbles of $P$.
  For every pair $C,D$ of distinct bubbles there is some point $y     \in K$
  with $y_{c} \neq y_{d}$ for $c \in C$ and $d \in D$.

  Let then $y$ be a general point in $K$. Then $y_c \neq y_d$ for
  $c,d$ in any two distinct bubbles. 
  By assumption, $y$ is in a cone $k(L) \sus K$.
  Then $\mathbf b(L) =\mathbf  b(P)$, so the dimensions of $k(L)$, $K$, and
  $k(P)$ are the same by Proposition \ref{pro:braid-kp}.
  The total preorder $L$ yields a total preorder on the bubbles,
  say
  \[ C_1 > C_2 > \cdots > C_{i-1} > C_i > C_{i+1} > C_{i+2} > \cdots > C_d. \]
  Suppose $C_i$ and $C_{i+1}$ are not comparable in $P$. Switching
  $C_i$ and $C_{i+1}$, there is then
  a linear extension $L^\prime$ of $P$ given by
  \[ C_1 > C_2 > \cdots > C_{i-1} > C_{i+1} > C_{i} > C_{i+2} > \cdots > C_d. \]
  We claim that $k(L^\prime) \sus K$. Since $k(L)$ is a closed cone, there
  is a point $\overline{y} \in k(L)$ with
  \[ \overline{y}_{C_1} >  \cdots > \overline{y}_{C_{i-1}} >
    \overline{y}_{C_i} =   \overline{y}_{C_{i+1}} >  \overline{y}_{C_{i+2}}
    > \cdots >  \overline{y}_{C_d}. \]
  There is also a point $x \in K$ with $x_{C_{i+1}} > x_{C_i}$.
  A point $y^\prime = \epsilon x + \overline{y}$ is then in $K$ and for
  $\epsilon$ sufficiently small
  \[ y^\prime_{C_1} >  \cdots > y^\prime_{C_{i-1}} >
    y^\prime_{C_{i+1}} >   y^\prime_{C_{i}} >  y^\prime_{C_{i+2}}
    > \cdots >  y^\prime_{C_d}. \]
  Then $K$ intersects the interior of $k(L^\prime)$ and so $k(L^\prime)$ is
  contained in $K$. We may in this way continue and switch bubbles, such
  that for every linear extension $\tilde{L}$ of $P$ we have $k(\tilde{L})
  \sus K$. Hence the union of all such cones, over the linear extensions
  of $P$, are in $K$ and since this union is $k(P)$, we get $K = k(P)$. 
\end{proof}

\begin{lemma} Let $C = k(P)$ and $D = k(Q)$ be braid cones.
  \begin{itemize}
  \item[a.] $C \cap D$ is the braid cone $k(P \vee Q)$.
  \item[b.] If $D \sus C$, then $C+ (-D)$ is the braid cone
    $k(P \wedge Q^{\op})$.
  \end{itemize}
\end{lemma}

\begin{remark}
  If $C$ and $D$ are braid cones, it is not necessarily true that $C + D$
  is a braid cone.
\end{remark}

\begin{proof} Part a is clear, as $C \cap D$ is cut out by the union
  \[ \{x_i \leq x_j \, | \, i \leq_P j\} \cup \{ x_i \leq x_j \, | \,
    i \leq_Q j \}. \]
  For part b, since $P \leq Q$, we have
  \[ P \wedge Q^{\op} = (P \wedge Q) \wedge Q^{\op} =
    P \wedge (Q \wedge Q^\op) = P \wedge Q^{\circ}. \]
 Furthermore $-D = k(Q^\op)$ and
  $C + (-D) \sus k(P \wedge Q^{\op})$,
  and so
  \[ k(P) + k(Q^{\circ}) = C + (D + (-D)) = C + (-D) \sus k(P \wedge Q^{\op}) . \]
We now show the opposite inclusion. 
  We may find a filtration $\emptyset = D_0 \sus D_1 \sus \cdots \sus D_m = Q$
  of down-sets of $Q$ such that each  $D_{i+1} \backslash D_i$ is a bubble
  of $Q$. As $P \leq Q$, each $D_i$ is also a down-set of $P$.
  Let $f \in k(P \wedge Q^{\circ})$. We construct a $g : P^\op \pil \RR$
  by successive 
  $g_i : D_i^\op \pil \RR$ as follows. To extend $g_i$ to $g_{i+1}$, define
  it as $f$ restricted to the $Q$-bubble, $D_{i+1} \backslash D_i$,
  but add a common large negative value to the $f(p)$ for $p$'s in the bubble,
  so as to make $g_i$ order-preserving on $D_{i+1}^\op$. As the $D_i$
  are all down-sets of $P$ this procedure will give an order-preserving
  $g$ in $k(P)$. Now we may add an element in $k(Q^\circ)$ to $g$ and this
  will give $f$. 
\end{proof}


\begin{figure}
\begin{center}
\begin{tikzpicture}[inner sep=1pt, scale=1.2]
\draw [help lines, white] (-2,-1) grid (3,2);


\node (rr) at (-1.9 , 0.5) {$R:$};
\node (aa2) at (-1.3,0.8) {};
\node (cc2) at (-1.6, 0.3) {};
\node (bb2) at (-1,0.3) {};
\node (dd2) at (-1.3, -0.2) {};

\fill (aa2) circle (0.06);
\fill (cc2) circle (0.06);
\fill (bb2) circle (0.06);
\fill (dd2) circle (0.06);

\draw (aa2)--(bb2);
\draw (cc2)--(dd2);


\node (pp) at (-1.9+1.8, 0.5) {$P:$};
\node (aa) at (-1.3+1.8,0.8) {};
\node (cc) at (-1.6+1.8, 0.3) {};
\node (bb) at (-1+1.8,0.3) {};
\node (dd) at (-1.3+1.8, -0.2) {};

\fill (aa) circle (0.06);
\fill (cc) circle (0.06);
\fill (bb) circle (0.06);
\fill (dd) circle (0.06);

\draw (aa)--(cc); 
\draw (aa)--(bb);
\draw (bb)--(dd);
\draw (cc)--(dd);


\node (qq) at (-1.9 +3.6, 0.5) {$Q:$};
\node (aa3) at (-1.25+3.6,0.75) {};
\node (cc3) at (-1.55+3.6, 0.05) {};
\node (bb3) at (-1.05+3.6,0.55) {};
\node (dd3) at (-1.35+3.6, -0.15) {};

\filldraw[gray,thin] (aa3) circle (0.05);
\filldraw[gray,thin] (cc3) circle (0.05);
\filldraw[gray,thin] (bb3) circle (0.05);
\filldraw[gray,thin] (dd3) circle (0.05);

\node (ss3) at (-1.15+3.6,0.65) {};
\draw (ss3) circle (0.24);
\node (tt3) at (-1.45+3.6,-0.05) {};
\draw (tt3) circle (0.24);


\draw (-1.22+3.6,0.41)--(-1.39+3.6, 0.19);
\end{tikzpicture},
\qquad \qquad
\begin{tikzpicture}[inner sep=1pt, scale=1.2]
\draw [help lines, white] (-2,-1) grid (3,2);


\node (pp) at (-1.9, 0.5) {$P:$};
\node (aa) at (-1.3,0.8) {};
\node (cc) at (-1.6, 0.3) {};
\node (bb) at (-1,0.3) {};
\node (dd) at (-1.3, -0.2) {};

\fill (aa) circle (0.06);
\fill (cc) circle (0.06);
\fill (bb) circle (0.06);
\fill (dd) circle (0.06);

\draw (aa)--(cc); 
\draw (aa)--(bb);
\draw (bb)--(dd);
\draw (cc)--(dd);


\node (qq) at (-1.9 +1.8, 0.5) {$Q:$};
\node (aa3) at (-1.3+1.8,0.75) {};
\node (cc3) at (-1.4+1.8, 0.3) {};
\node (bb3) at (-1.2+1.8,0.3) {};
\node (dd3) at (-1.3+1.8, -0.15) {};

\fill (aa3) circle (0.05);
\filldraw[gray,thin] (cc3) circle (0.05);
\filldraw[gray,thin] (bb3) circle (0.05);
\fill (dd3) circle (0.05);

\node (ss3) at (-1.3+1.8,0.3) {};
\draw (ss3) circle (0.18);


\draw (aa3)--(-1.3+1.8, 0.5);
\draw (-1.3+1.8,0.1)--(dd3);

\end{tikzpicture}

\caption{$R \lhd P$ and corresponding $P \btl Q$, \hskip 5mm
$P \preceq Q$ but not $P \btl Q$ }
\label{Fig-TT}
\end{center}
\end{figure}
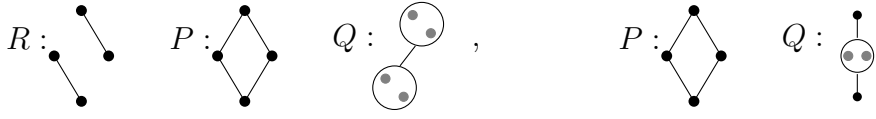

For a braid cone $C = k(P)$ we may let $\subc(P)$ be the preorders
$R \leq P$ and $\supc(P)$ be the preorders $R \leq Q$.
The Galois connection \eqref{eq:cones-galois} thus induces a Galois
connection
\begin{align} \label{eq:braid-galois} \supc(P) & \bihom{\Phi}{\Gamma}
                                                 \subc(P) \\
  Q & \mapsto P \wedge Q^{\op}  \notag\\
  P \vee R^{\op} & \mapsfrom R. \notag
  \end{align}  
  By Lemma \ref{lem:cones-face}, the image of $\Gamma$ is  the set of
  $Q$ such that $k(Q)$ is a face of $k(P)$. This description is given
  in \cite[Section 3]{PRW2009}, and $Q$ is called a contraction of $P$.
  We shall write $P \btl Q$ for such $Q$.
  The image of $\Phi$ are the $R$ which may be written as $P \wedge Q^{\op}$
  for some superpreorder $P \leq Q$. We write $R \wtl P$.
  Such $R$ go by several names in the literature:
  \begin{itemize}
  \item Partitions of $P$ which are connected and compatible,
    \cite{Sta1986},
  \item Positive subposets in \cite[Sec.3.4]{AA2023} (they only consider
    posets),
  \item Admissible in \cite{FFM2017}.
 \end{itemize}

\section{Submodular and modular functions}

\label{sec:submod}

The inequalities defining extended generalized permutahedra (EGP),
are derived from extended submodular functions. In fact there is a bijection
between these structures. Also an extended  submodular function induces a
preorder, which plays a significant role for the geometry of EGP's.
EGP's which are affine cones correspond to the subclass of modular functions.

\subsection{Submodular functions}
\label{subsec:pre-bf}

Denote $\hRR = \RR \cup \{\infty \}$. For $I$ a finite set, let
$\pow(I)$ be the power set, i.e. the set of all subsets of $I$. An
{\it extended Boolean function}
on $I$ is an arbitrary function $z : \pow(I) \to \hRR$ such
that $z(\emptyset) = 0$ and $z(I)< \infty$. Write $\bool(I)$ for the set of
extended Boolean functions
on $I$. For $S \sus I$ with $z(S) < \infty$,
the {\it restriction} of $z$ to $S$ is
\[ z_{|S} : \pow(S) \to \hRR, \, \text{  given by }
  z_{|S}(U) = z(U), \quad U \sus S.\]
The {\it corestriction} of $z$ is
\[z_{/S} : \pow(I \backslash S) \to \hRR, \, \text{ given by }
  z_{/S}(U) = z(S \cup U) - z(S), \quad U \sus I \backslash S. \]
If $S \sus T \sus I$ let $z_{T/S}$ be $(z_{|T})_{/S}$.

\medskip
Let $u \in \bool(S)$ and $v \in \bool(T)$ and $I = S \sqcup T$
the disjoint union.
Their product $u \cdot v \in \bool(I)$ is the function given by
\[ u \cdot v(E) = u(E \cap S) + v(E \cap T), \quad E \sus I. \]

\noindent The function $z \in \bool(I)$ is {\it submodular}
if for $S,T \sus I$:
\begin{equation} \label{eq:notions-submod}
  z (S) + z(T) \geq z(S \cup T) + z(S \cap T).
  \end{equation}
  Write $\smod(I)$ for the set of submodular functions on $I$.
  An equivalent description is that a Boolean function is submodular if and
  only if 
  given any $I \supseteq B \supseteq A$ and varying
  $S \sus I \backslash B$, then:
  \[ z(B \cup S) - z(A \cup S) \]
  is a weakly decreasing function as $S$ increases for the inclusion relation
  (with the convention that $\infty - \infty$ is any value).

\subsection{From submodular functions to
  topologies and preorders}  \label{subsec:prod-subpretop}
The following is a simple and significant observation.
Let $I$ be
finite and given a submodular function $z$. Let
\[\top(z) = \{ S \sus I \, |\, z(S) < \infty \}. \]
This is a topology on $I$. This is because it is closed under unions
and intersection due to the submodular inequality \eqref{eq:notions-submod}.
By the Alexandroff correspondence, see for instance
\cite[Sec.3]{F-order}, a finite topology on $I$ is equivalent to give
preorder on $I$: A set $S \sus I$ is an open set in the topology iff
$S$ is a down-set for the preorder.
So we have maps:
\[ \smod(I)\, \,  \substack{ \top \\ \pil \\ \pil \\ \pre} \, \, 
  \begin{cases} \Top(I) \\ \Pre(I) \end{cases} \]
which via the identification of $\Top(I)^{\op}$ and $\Pre(I)$ is
the same map. So $z(S)$ is finite if and only if $S$ is a down-set
in $\pre(z)$.

Conversely, there is a map 
sending a preorder $P$ to the submodular function
\[ \low_P : \pow(I) \pil \hat{\RR}, \quad \text{where } \low_P(S) =
  \begin{cases} 0, & S \text{ a down-set of } P \\
    \infty, & S \text{ not down-set of } P
  \end{cases}. \]
It is immediate that submodular functions $z : \pow(I) \pil \{ 0, \infty\}$ are
in bijection with preorders on $I$, \cite[Thm. 3.4.9]{AA2023}.
  
\subsection{Modular functions}The class of modular functions $z$ on $I$, i.e. those such that we have
equality in \eqref{eq:notions-submod}, is particularly nice.

A function $z : \pow(I) \pil \hRR$ is {\it modular} iff
for every $A, B \sus I$:
\begin{equation} \label{eq:mod-def}
  z(A) + z(B) = z(A \cap B) + z(A \cup B).
  \end{equation}
Let $P = \pre(z)$. 
For the modular function $z$, each bubble $C \in {\mathbf b}(P)$,
generates a down-set $D = \dai{P}C$. Let $E$ be the down-set
$(\dai{P}) C \backslash C$.
We get a function
\[ v : {\mathbf b}(P) \pil \RR, \quad v(C) = z_{D \backslash E}(C). \]

Given a preorder $P$ on $I$, let $\mod_P(I)$ be the set of modular functions
$z$ with $P = \pre(z)$. The above gives a map
{
\begin{equation} \label{eq:mod-bpr}
  \mod_P(I) \pil \RR^{\mathbf b(P)}, \quad z \mapsto v
\end{equation}
}

\begin{proposition} \label{pro:submod-mod}
The association \eqref{eq:mod-bpr} above 
is a bijection.
In particular if $P$ is a poset (which is equivalent to $\mathbf b(P)$
being the family of one-element subsets of $I$),
modular functions in $\mod_P(I)$ are in bijection with $\RR^I$.
{Modular} functions $z$, may thus be given by pairs $(P,v)$
where $P$ is a preorder and $v \in \RR^{\mathbf b(P)}$. For
a down-set $S$ of $P$:
\[ z(S) = \sum_{C \in {\mathbf b(S)}} v(C). \]
\end{proposition}

\begin{proof}
  Given a map $v : \mathbf b(P) \pil \RR$, we show there is a uniquely defined
  modular function $z : \pow(I) \pil \RR$ with $z(D)$ finite iff
  $D$ is a down-set of $I$, which gives $v$ by the correspondence
  above. Let $D$ be a down-set of $P$. We define inductively $z(D)$ by
  the size of $D$.
 Clearly minimal $D$'s are minimal bubbles, and 
 $z(D) = v(D)$.
 If $D$ is of the form $\dai{P} C$ for a bubble $C$, we define
 inductively
  \[ z(D) = v(C) + z(D \backslash C). \]
 
  Otherwise, we may
  write $D = A_1 \cup B_1$ for two down-sets strictly smaller than $D$.
  Since we must have
  \[ z(A_1 \cup B_1) + z(A_1 \cap B_1) = z(A_1) + z(B_1), \]
  this determines $z(D)$. But we must show this is independent of the
  decomposition. So let $D = A_2 \cup B_2$ be another decomposition.
  Then $z(D)$ is defined respectively as:
  \begin{equation} \label{eq:mod-zD}
    z(A_1) + z(B_1) - z(A_1 \cap B_1), \quad
    z(A_2) + z(B_2) - z(A_2 \cap B_2).
    \end{equation} 
  Expanding the first expression in \eqref{eq:mod-zD} by:
  \[ A_1 = (A_1 \cap A_2) \cup (A_1 \cap B_2), \quad
    B_1 = (B_1 \cap B_2) \cup (B_1 \cap A_2), \]
  we get
  \begin{align*}
\big ((z(A_1 \cap A_2) + z(A_1 \cap B_2) & -  z(A_1 \cap A_2 \cap B_2) \big )
+ \big ( z(B_1 \cap A_2) + z(B_1 \cap B_2) - z(B_1 \cap B_2 \cap A_2) \big ) \\
                  & - \big ( z(A_1 \cap B_1 \cap A_2) +
        z(A_1 \cap B_1 \cap B_2) - z(A_1 \cap A_2 \cap B_1 \cap B_2) \big ).
  \end{align*}
  Expanding the second expression of \eqref{eq:mod-zD}, it is seen, by symmetry,
  to give the same as above.
\end{proof}

\section{Extended generalized permutahedra}

\label{sec:egp}

By linear inequalities associated to submodular functions,
they define convex polyhedra. If 
the submodular function only takes finite values, these are polytopes:
the generalized
permutahedra. Allowing $\infty$ as value, we get extended generalized
permutahedra. We recall their definition and basic properties.
In particular, we give the submodular functions of faces
and tangent cones of extended generalized permutahedra,
Theorems \ref{thm:egp-face} and  \ref{thm:egp-cone}. To our knowledge these
have not appeared explicitly in the literature before.

We base our development here on Facts \ref{fact:egp-prodsum}, \ref{fact:egp-1S},
and \ref{fact:egp-braid} which are found in \cite{AA2023}. An original,
detailed and rich source is \cite{Fuji}.

\subsection{Definition and basic facts}
Again $I$ is a finite set, and $\RR I$ the vector space with
basis $I$, and we write $e_i$ for the basis element $i$.
The elements of $\RR I$ are thus sums $\sum a_i e_i$ where
$a_i \in \RR$. We may also write $(a_i)_{i \in I}$.

The \textsl{extended generalized permutahedron (EGP)} associated to a submodular
function $z:\pow(I)\to \hRR$ 
is the polyhedron $\egp(z)$ consisting of points $x = (x_i)_{i \in I} \in \RR I$
such that:
\begin{equation} \label{eq:egp-def}
  i. \,   \sum_{i\in I}x_i=z(I), \quad ii. \,
  \sum_{i\in A}x_i\le z(A) \hbox{ for each }
    A\sus I \hbox{ with } z(A)< \infty.
  \end{equation}

This is a convex polyhedron. It is a polytope, a
{\it generalized permutahedron},
if and only if $z:\pow(I)\to\mathbb R$, i.e. $z$ only takes finite values.
When $x \in \RR I$ and $A \sus I$,
we write for short $x_A = \sum_{i \in A} x_i$. 

\begin{example}
  Given a strictly decreasing sequence
  $\ell_1 > \ell_2 > \cdots  > \ell_n$ in $\RR_{\geq 0}$, for $S \sus I$ let 
  $z(S) = \sum_{i \leq |S|} \ell_i$. This is a submodular function,
  and the associated GP is the permutahedron, the convex hull of
  the vertices $(\ell_{\pi(1)}, \ldots, \ell_{\pi(n)})$ where $\pi$ runs
  through all permutations of $\{1,2, \ldots, n\}$. 
\end{example}

We recall basic facts on extended generalized permutahedra (EGP's).

\begin{fact} \label{fact:egp-prodsum}
  Let $I = S \sqcup T$.
  If $z \in \smod(I)$ decomposes as a product $z = u \cdot v$,
  where $u \in \smod(S)$ and
  $v \in \smod(T)$, then
  $\egperm(z)$ on $\RR I$ is the product
  $\egperm(u) \times \egperm(v) \sus \RR S \times \RR T$.
  This is \cite[Thm. 3.1.8]{AA2023}.
\end{fact}

\begin{fact} \label{fact:egp-1S}
  When $z(S) < \infty$, the equation $x_S = z(S)$ defines a face of
  $\egperm(z)$, which is the product
  \[ \egperm(z_{|S}) \times \egperm (z_{/S}). \]
  It is the face of $\egperm(z)$ associated to the direction $1_S$.
 
  When  $z(S) = \infty$, then $1_S$ does not take a maximum
  value on $\egperm(z)$ (it takes arbitrarily large values).
  This is Proposition 1.4.2 and Theorem 3.1.8 in \cite{AA2023},
  and the last statement of Theorem 3.1.6 of \cite{AA2023}
  (which uses Section 3 in \cite{Fuji}).
\end{fact}

\begin{fact}  \label{fact:egp-braid}
 The normal fan of an extended generalized permutahedron $\egperm(z)$ is
  a coarsening of the induced subfan of the braid arrangement on
  a braid cone. This is Theorem 3.1.6 parts (1) and (2) in \cite{AA2023}.
\end{fact}

\ignore{    
\begin{proposition}    
  The braid cone of Fact \ref{fact:egp-braid} is $k(\pre(z))$
  (where $\pre(z)$ is the preorder defined in Subsection
  \ref{subsec:prod-subpretop}).
\end{proposition}

\begin{proof} By Proposition \ref{pro:braid-K} the braid cone
  is $k(P)$ for some preorder $P$. 
  Down-sets $S$ of $P$ corresponds to the bounded directions $1_S$
  of $z$. By Fact \ref{fact:egp-1S} these are those $S$ such that
  $z(S) < \infty$, and so $P = \pre(z)$.
\end{proof}
}

\begin{proposition} \label{pro:egp-normalcone}
  A face $F$ of $\egperm(z)$ has normal cone
  $k(P)$ for some preorder $P$. For $S \sus I$, the following are equivalent:

  \begin{enumerate}[a)]
\item $F$ is contained in the face where $1_S$ attains its maximum. 
\item $S$ is a down-set of $P$,
\item The equation $x_S = z(S)$ holds for every $x \in F$,
\end{enumerate}

In particular, the $S$ such that $x_S = z(S)$ holds on $F$
are precisely the down-sets of $P$. 
\end{proposition}

\begin{proof} That the normal cone of $F$ is a braid cone $k(P)$ follows
  by  Fact \ref{fact:egp-braid} and Proposition \ref{pro:braid-K}.
  That a) and b) are equivalent has already been observed in
  Section \ref{sec:braid}, before Proposition \ref{pro:braid-cone}.

  We know that $x_S \leq z(S)$ on $\egperm(z)$, and equality may be attained
  by Fact \ref{fact:egp-1S}.
  So the maximum value of $1_S$ is $z(S)$. This shows a) implies c).
Conversely, if $x_S = z(S)$ on $F$, then $1_S$ attains maximal value
  on $F$, so c) implies a).
\end{proof}

We get the following from \cite[Prop. 3.5.(2)]{PRW2009}:

\begin{corollary} \label{cor:egp-PQ}
  Let $F,G$ be faces of $\egperm(z)$, corresponding to
  normal fans $k(P)$ and $k(Q)$. Then $F \sus G$ iff $P \btl Q$, i.e.
  $Q$ is a contraction of $P$.
\end{corollary}

\begin{proof} That $F$ is a face of $G$, corresponds to $k(Q)$ being
  a face of $k(P)$. By Subsection \ref{subsec:braidcone}, this
  means $P \btl Q$.
\end{proof}


\newcommand{\cba}{
  \begin{tikzpicture}[scale=0.7, vertices/.style={draw, fill=black, circle, inner sep=1.5pt}]

\node [vertices] (31) at (-0.5,0) {};
\node [vertices] (32) at (-0.5,0.5) {};
\node [vertices] (33) at (-0.5,1) {};

\node[anchor = east] at (31) {$c$ \hskip 4mm {}};
\node[anchor = east] at (32) {$b$ \hskip 4mm {}};
\node[anchor = east] at (33) {$a$ \hskip 4mm {}};
\draw (31)--(32);
\draw (31)--(33); 

\end{tikzpicture}
}


\newcommand{\cab}{
  \begin{tikzpicture}[scale=0.7, vertices/.style={draw, fill=black, circle, inner sep=1.5pt}]

\node [vertices] (31) at (-0.5,0) {};
\node [vertices] (32) at (-0.5,0.5) {};
\node [vertices] (33) at (-0.5,1) {};

\node[anchor = east] at (31) {$c$ \hskip 4mm {}};
\node[anchor = east] at (32) {$a$ \hskip 4mm {}};
\node[anchor = east] at (33) {$b$ \hskip 4mm {}};
\draw (31)--(32);
\draw (31)--(33); 

\end{tikzpicture}
}


\newcommand{\acb}{
  \begin{tikzpicture}[scale=0.7, vertices/.style={draw, fill=black, circle, inner sep=1.5pt}]

\node [vertices] (31) at (-0.5,0) {};
\node [vertices] (32) at (-0.5,0.5) {};
\node [vertices] (33) at (-0.5,1) {};

\node[anchor = east] at (31) {$a$ \hskip 4mm {}};
\node[anchor = east] at (32) {$c$ \hskip 4mm {}};
\node[anchor = east] at (33) {$b$ \hskip 4mm {}};
\draw (31)--(32);
\draw (31)--(33); 

\end{tikzpicture}
}


\newcommand{\abc}{
  \begin{tikzpicture}[scale=0.7, vertices/.style={draw, fill=black, circle, inner sep=1.5pt}]

\node [vertices] (31) at (-0.5,0) {};
\node [vertices] (32) at (-0.5,0.5) {};
\node [vertices] (33) at (-0.5,1) {};

\node[anchor = east] at (31) {$a$ \hskip 4mm {}};
\node[anchor = east] at (32) {$b$ \hskip 4mm {}};
\node[anchor = east] at (33) {$c$ \hskip 4mm {}};
\draw (31)--(32);
\draw (31)--(33); 

\end{tikzpicture}
}

\newcommand{\bac}{
  \begin{tikzpicture}[scale=0.7, vertices/.style={draw, fill=black, circle, inner sep=1.5pt}]

\node [vertices] (31) at (-0.5,0) {};
\node [vertices] (32) at (-0.5,0.5) {};
\node [vertices] (33) at (-0.5,1) {};

\node[anchor = east] at (31) {$b$ \hskip 4mm {}};
\node[anchor = east] at (32) {$a$ \hskip 4mm {}};
\node[anchor = east] at (33) {$c$ \hskip 4mm {}};
\draw (31)--(32);
\draw (31)--(33); 

\end{tikzpicture}
}


\newcommand{\bca}{
  \begin{tikzpicture}[scale=0.7, vertices/.style={draw, fill=black, circle, inner sep=1.5pt}]

\node [vertices] (31) at (-0.5,0) {};
\node [vertices] (32) at (-0.5,0.5) {};
\node [vertices] (33) at (-0.5,1) {};

\node[anchor = east] at (31) {$b$ \hskip 4mm {}};
\node[anchor = east] at (32) {$c$ \hskip 4mm {}};
\node[anchor = east] at (33) {$a$ \hskip 4mm {}};
\draw (31)--(32);
\draw (31)--(33); 

\end{tikzpicture}
}


\newcommand{\csab}{
  \begin{tikzpicture}[scale=0.7, vertices/.style={draw, fill=black, circle, inner sep=1pt}]

\node [vertices] (31) at (-0.5,-0.2) {};
\node [vertices, gray] (32) at (-0.33,0.7) {};
\node [vertices, gray] (33) at (-0.67,0.7) {};

\node[anchor = east] at (31) {$c$ \hskip 0.1pt {}};
\node[anchor = west] at (32) {$\, b${}};
\node[anchor = east] at (33) {$a\, $};

\node (ss3) at (-0.5, 0.7) {};
\draw (ss3) ellipse (0.63 and 0.44);


\draw (31)--(-0.45, 0.25);

\end{tikzpicture}
}

\newcommand{\asbc}{
  \begin{tikzpicture}[scale=0.7, vertices/.style={draw, fill=black, circle, inner sep=1pt}]

\node [vertices] (31) at (-0.5,-0.2) {};
\node [vertices, gray] (32) at (-0.33,0.7) {};
\node [vertices, gray] (33) at (-0.67,0.7) {};

\node[anchor = east] at (31) {$a$ \hskip 0.1pt {}};
\node[anchor = west] at (32) {$\, c${}};
\node[anchor = east] at (33) {$b\, $};

\node (ss3) at (-0.5, 0.7) {};
\draw (ss3) ellipse (0.63 and 0.44);


\draw (31)--(-0.44, 0.32);

\end{tikzpicture}
}

\newcommand{\bsac}{
  \begin{tikzpicture}[scale=0.7, vertices/.style={draw, fill=black, circle, inner sep=1pt}]

\node [vertices] (31) at (-0.5,-0.2) {};
\node [vertices, gray] (32) at (-0.33,0.7) {};
\node [vertices, gray] (33) at (-0.67,0.7) {};

\node[anchor = east] at (31) {$b$ \hskip 0.1pt {}};
\node[anchor = west] at (32) {$\, c${}};
\node[anchor = east] at (33) {$a\, $};

\node (ss3) at (-0.5, 0.7) {};
\draw (ss3) ellipse (0.63 and 0.44);

\draw (31)--(-0.57, 0.27);

\end{tikzpicture}
}


\newcommand{\absc}{
  \begin{tikzpicture}[scale=0.7, vertices/.style={draw, fill=black, circle, inner sep=1pt}]

\node [vertices, gray] (31) at (-0.33,-0.2) {};
\node [vertices, gray] (32) at (-0.67,-0.2) {};
\node [vertices] (33) at (-0.5,0.7) {};

\node[anchor = west] at (31) {$\, b$ };
\node[anchor = east] at (32) {$a \,$ };
\node[anchor = east] at (33) {$c$ \hskip 0.1pt {}};

\node (ss3) at (-0.5, -0.2) {};
\draw (ss3) ellipse (0.63 and 0.44);


\draw (33)--(-0.57, 0.25);

\end{tikzpicture}
}

\newcommand{\bcsa}{
  \begin{tikzpicture}[scale=0.7, vertices/.style={draw, fill=black, circle, inner sep=1pt}]

\node [vertices, gray] (31) at (-0.33,-0.2) {};
\node [vertices, gray] (32) at (-0.67,-0.2) {};
\node [vertices] (33) at (-0.5,0.7) {};

\node[anchor = west] at (31) {$\, c$ };
\node[anchor = east] at (32) {$b \,$ };
\node[anchor = east] at (33) {$a$ \hskip 0.1pt {}};

\node (ss3) at (-0.5, -0.2) {};
\draw (ss3) ellipse (0.63 and 0.44);


\draw (33)--(-0.57, 0.20);

\end{tikzpicture}
}

\newcommand{\acsb}{
  \begin{tikzpicture}[scale=0.7, vertices/.style={draw, fill=black, circle, inner sep=1pt}]

\node [vertices, gray] (31) at (-0.33,-0.2) {};
\node [vertices, gray] (32) at (-0.67,-0.2) {};
\node [vertices] (33) at (-0.5,0.7) {};

\node[anchor = west] at (31) {$\, c$ };
\node[anchor = east] at (32) {$a \,$ };
\node[anchor = east] at (33) {$b$ \hskip 0.1pt {}};

\node (ss3) at (-0.5, -0.2) {};
\draw (ss3) ellipse (0.63 and 0.44);


\draw (33)--(-0.42, 0.20);

\end{tikzpicture}
}


\newcommand{\abcss}{
  \begin{tikzpicture}[scale=0.7, vertices/.style={draw, fill=black, circle, inner sep=1pt}]

\node [vertices, gray] (31) at (-0.33,-0.2) {};
\node [vertices, gray] (32) at (-0.67,-0.2) {};
\node [vertices, gray] (33) at (-0.5, 0.03) {};

\node[anchor = west] at (31) {$\, c$ };
\node[anchor = east] at (32) {$b \,$ };
\node[anchor = south east] at (33) {$a$};

\node (ss3) at (-0.5, -0.13) {};
\draw (ss3) ellipse (0.60 and 0.60);


\end{tikzpicture}
}


\begin{figure}
\begin{center}
\begin{tikzpicture}[inner sep=1pt, scale=0.7]

\coordinate (a) at (0,0) ; 
\coordinate (b) at (3, 1.71);  
\coordinate (c) at (3, 5.15);  
\coordinate (d) at (0, 6.86);  
\coordinate (e) at (-3, 5.15);  
\coordinate (f) at (-3, 1.71);  

            \draw[black, thick, fill=tol1, fill opacity = 0.1] 
            (a) -- (b)  -- (c) -- (d) -- (e) -- (f) -- (a);
            \draw[black, fill] (a)     circle (2pt);
            \draw[black, fill] (b)     circle (2pt);
            \draw[black, fill] (c)     circle (2pt);
            \draw[black, fill] (d)     circle (2pt);
            \draw[black, fill] (e)     circle (2pt);
            \draw[black, fill] (f)     circle (2pt);

\node[] at (a) [anchor = north east] {$(1,2,3)$};
\node[] at (b) [anchor = north west] {$(2,1,3)$};
\node[] at (c) [anchor = south west] {$(3,1,2)$};
\node[] at (d) [anchor = south west] {$(3,2,1)$};
\node[] at (e) [anchor = south east] {$(2,3,1)$};
\node[] at (f) [anchor = north east] {$(1,3,2)$};
      
\end{tikzpicture}
\quad   
\begin{tikzpicture}[inner sep=1pt, scale=0.7]

\coordinate (a) at (0,0) ; 
\coordinate (b) at (3, 1.71);  
\coordinate (c) at (3, 5.15);  
\coordinate (d) at (0, 6.86);  
\coordinate (e) at (-3, 5.15);  
\coordinate (f) at (-3, 1.71);  

\coordinate (x) at (-1.5,0.855) ; 
\coordinate (y) at (1.5, 0.855);  
\coordinate (z) at (3, 3.43);  
\coordinate (t) at (1.5, 6.005);  
\coordinate (u) at (-1.5, 6.005);  
\coordinate (v) at (-3, 3.43);  

\coordinate (s) at (0, 3.43);

            \draw[black, thick, fill=tol1, fill opacity = 0.1] 
            (a) -- (b)  -- (c) -- (d) -- (e) -- (f) -- (a);
            \draw[black, fill] (a)     circle (2pt);
            \draw[black, fill] (b)     circle (2pt);
            \draw[black, fill] (c)     circle (2pt);
            \draw[black, fill] (d)     circle (2pt);
            \draw[black, fill] (e)     circle (2pt);
            \draw[black, fill] (f)     circle (2pt);

\node[] at (a) [anchor = north east] {$\cba$ \hskip 5mm {}};
\node[] at (b) [anchor = north west] {$\cab$ \hskip 5mm {}};
\node[] at (c) [anchor = south west] {$\acb$ \hskip 5mm {}};
\node[] at (d) [anchor = south west] {$\abc$ \hskip 5mm {}};
\node[] at (e) [anchor = south east] {$\bac$ \hskip 5mm {}};
\node[] at (f) [anchor = north east] {$\bca$ \hskip 5mm {}};
      
\node[] at (x) [anchor = north east] {$\bcsa$};
\node[] at (y) [anchor = north west] {$\csab$};
\node[] at (z) [anchor = west] {$\acsb$};
\node[] at (t) [anchor = south west] {$\asbc$};
\node[] at (u) [anchor = south east] {$\absc$};
\node[] at (v) [anchor = east] {$\bsac$};

\node[] at (s) {$\abcss$};
\end{tikzpicture}
          \end{center}
          \caption{The two-dimensional permutahedron and the preorders
          associated to its faces}
        \label{fig:toperm-conform}
    \end{figure}
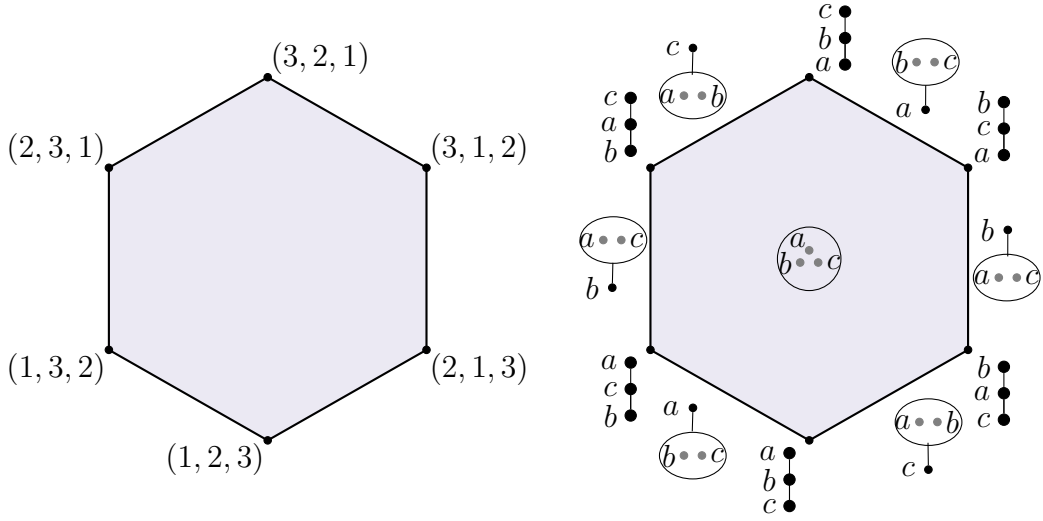    

\begin{example} \label{ex:bij-pent} Figure
  \ref{fig:toperm-conform} shows the two-dimensional permutahedron
  and the preorders associated to its faces. 
  Figure \ref{fig:togenperm-conform} shows the
  generalized permutahedron, a pentagon, of the submodular function defined on
  $\pow(\{a,b,c\})$ by 
  \[ z(a) = z(b) = z(c) = 3, \quad z(ab) = z(bc) = 5, \quad
    z(ac) = z(abc) = 6, \]
  and the corresponding preorders.
\end{example}


\newcommand{\recba}{
  \begin{tikzpicture}[scale=0.7, vertices/.style={draw, fill=black, circle, inner sep=1.5pt}]

\node [vertices] (31) at (-0.5,0) {};
\node [vertices] (32) at (-0.5,0.5) {};
\node [vertices] (33) at (-0.5,1) {};

\node[anchor = east] at (31) {$c$ \hskip 4mm {}};
\node[anchor = east] at (32) {$b$ \hskip 4mm {}};
\node[anchor = east] at (33) {$a$ \hskip 4mm {}};
\draw (31)--(32);
\draw (31)--(33); 

\end{tikzpicture}
}


\renewcommand{\cab}{
  \begin{tikzpicture}[scale=0.7, vertices/.style={draw, fill=black, circle, inner sep=1.5pt}]

\node [vertices] (31) at (-0.5,0) {};
\node [vertices] (32) at (-0.5,0.5) {};
\node [vertices] (33) at (-0.5,1) {};

\node[anchor = east] at (31) {$c$ \hskip 4mm {}};
\node[anchor = east] at (32) {$a$ \hskip 4mm {}};
\node[anchor = east] at (33) {$b$ \hskip 4mm {}};
\draw (31)--(32);
\draw (31)--(33); 

\end{tikzpicture}
}


\renewcommand{\acb}{
  \begin{tikzpicture}[scale=0.7, vertices/.style={draw, fill=black, circle, inner sep=1.5pt}]

\node [vertices] (31) at (-0.5,0) {};
\node [vertices] (32) at (-0.5,0.5) {};
\node [vertices] (33) at (-0.5,1) {};

\node[anchor = east] at (31) {$a$ \hskip 4mm {}};
\node[anchor = east] at (32) {$c$ \hskip 4mm {}};
\node[anchor = east] at (33) {$b$ \hskip 4mm {}};
\draw (31)--(32);
\draw (31)--(33); 

\end{tikzpicture}
}


\renewcommand{\abc}{
  \begin{tikzpicture}[scale=0.7, vertices/.style={draw, fill=black, circle, inner sep=1.5pt}]

\node [vertices] (31) at (-0.5,0) {};
\node [vertices] (32) at (-0.5,0.5) {};
\node [vertices] (33) at (-0.5,1) {};

\node[anchor = east] at (31) {$a$ \hskip 4mm {}};
\node[anchor = east] at (32) {$b$ \hskip 4mm {}};
\node[anchor = east] at (33) {$c$ \hskip 4mm {}};
\draw (31)--(32);
\draw (31)--(33); 

\end{tikzpicture}
}

\renewcommand{\bac}{
  \begin{tikzpicture}[scale=0.7, vertices/.style={draw, fill=black, circle, inner sep=1.5pt}]

\node [vertices] (31) at (-0.5,0) {};
\node [vertices] (32) at (-0.5,0.5) {};
\node [vertices] (33) at (-0.5,1) {};

\node[anchor = east] at (31) {$b$ \hskip 4mm {}};
\node[anchor = east] at (32) {$a$ \hskip 4mm {}};
\node[anchor = east] at (33) {$c$ \hskip 4mm {}};
\draw (31)--(32);
\draw (31)--(33); 

\end{tikzpicture}
}


\renewcommand{\bca}{
  \begin{tikzpicture}[scale=0.7, vertices/.style={draw, fill=black, circle, inner sep=1.5pt}]

\node [vertices] (31) at (-0.5,0) {};
\node [vertices] (32) at (-0.5,0.5) {};
\node [vertices] (33) at (-0.5,1) {};

\node[anchor = east] at (31) {$b$ \hskip 4mm {}};
\node[anchor = east] at (32) {$c$ \hskip 4mm {}};
\node[anchor = east] at (33) {$a$ \hskip 4mm {}};
\draw (31)--(32);
\draw (31)--(33); 

\end{tikzpicture}
}


\renewcommand{\csab}{
  \begin{tikzpicture}[scale=0.7, vertices/.style={draw, fill=black, circle, inner sep=1.5pt}]

\node [vertices] (31) at (-0.5,-0.2) {};
\node [vertices, gray] (32) at (-0.33,0.7) {};
\node [vertices, gray] (33) at (-0.67,0.7) {};

\node[anchor = east] at (31) {$c$ \hskip 0.1pt {}};
\node[anchor = west] at (32) {$\, b${}};
\node[anchor = east] at (33) {$a\, $};

\node (ss3) at (-0.5, 0.7) {};
\draw (ss3) ellipse (0.63 and 0.44);


\draw (31)--(-0.45, 0.25);

\end{tikzpicture}
}

\renewcommand{\asbc}{
  \begin{tikzpicture}[scale=0.7, vertices/.style={draw, fill=black, circle, inner sep=1.5pt}]

\node [vertices] (31) at (-0.5,-0.2) {};
\node [vertices, gray] (32) at (-0.33,0.7) {};
\node [vertices, gray] (33) at (-0.67,0.7) {};

\node[anchor = east] at (31) {$a$ \hskip 0.1pt {}};
\node[anchor = west] at (32) {$\, c${}};
\node[anchor = east] at (33) {$b\, $};

\node (ss3) at (-0.5, 0.7) {};
\draw (ss3) ellipse (0.63 and 0.44);


\draw (31)--(-0.44, 0.32);

\end{tikzpicture}
}

\renewcommand{\bsac}{
  \begin{tikzpicture}[scale=0.7, vertices/.style={draw, fill=black, circle, inner sep=1.5pt}]

\node [vertices] (31) at (-0.5,-0.2) {};
\node [vertices, gray] (32) at (-0.33,0.7) {};
\node [vertices, gray] (33) at (-0.67,0.7) {};

\node[anchor = east] at (31) {$b$ \hskip 0.1pt {}};
\node[anchor = west] at (32) {$\, c${}};
\node[anchor = east] at (33) {$a\, $};

\node (ss3) at (-0.5, 0.7) {};
\draw (ss3) ellipse (0.63 and 0.44);

\draw (31)--(-0.57, 0.27);

\end{tikzpicture}
}


\renewcommand{\absc}{
  \begin{tikzpicture}[scale=0.7, vertices/.style={draw, fill=black, circle, inner sep=1.5pt}]

\node [vertices, gray] (31) at (-0.33,-0.2) {};
\node [vertices, gray] (32) at (-0.67,-0.2) {};
\node [vertices] (33) at (-0.5,0.7) {};

\node[anchor = west] at (31) {$\, b$ };
\node[anchor = east] at (32) {$a \,$ };
\node[anchor = east] at (33) {$c$ \hskip 0.1pt {}};

\node (ss3) at (-0.5, -0.2) {};
\draw (ss3) ellipse (0.63 and 0.44);


\draw (33)--(-0.57, 0.25);

\end{tikzpicture}
}

\renewcommand{\bcsa}{
  \begin{tikzpicture}[scale=0.7, vertices/.style={draw, fill=black, circle, inner sep=1.5pt}]

\node [vertices, gray] (31) at (-0.33,-0.2) {};
\node [vertices, gray] (32) at (-0.67,-0.2) {};
\node [vertices] (33) at (-0.5,0.7) {};

\node[anchor = west] at (31) {$\, c$ };
\node[anchor = east] at (32) {$b \,$ };
\node[anchor = east] at (33) {$a$ \hskip 0.1pt {}};

\node (ss3) at (-0.5, -0.2) {};
\draw (ss3) ellipse (0.63 and 0.44);


\draw (33)--(-0.57, 0.20);

\end{tikzpicture}
}

\renewcommand{\acsb}{
  \begin{tikzpicture}[scale=0.7, vertices/.style={draw, fill=black, circle, inner sep=1.5pt}]

\node [vertices, gray] (31) at (-0.33,-0.2) {};
\node [vertices, gray] (32) at (-0.67,-0.2) {};
\node [vertices] (33) at (-0.5,0.7) {};

\node[anchor = west] at (31) {$\, c$ };
\node[anchor = east] at (32) {$a \,$ };
\node[anchor = east] at (33) {$b$ \hskip 0.1pt {}};

\node (ss3) at (-0.5, -0.2) {};
\draw (ss3) ellipse (0.63 and 0.44);


\draw (33)--(-0.42, 0.20);

\end{tikzpicture}
}


\renewcommand{\abcss}{
  \begin{tikzpicture}[scale=0.7, vertices/.style={draw, fill=black, circle, inner sep=1.5pt}]

\node [vertices, gray] (31) at (-0.33,-0.2) {};
\node [vertices, gray] (32) at (-0.67,-0.2) {};
\node [vertices, gray] (33) at (-0.5, 0.03) {};

\node[anchor = west] at (31) {$\, c$ };
\node[anchor = east] at (32) {$b \,$ };
\node[anchor = south east] at (33) {$a$};

\node (ss3) at (-0.5, -0.13) {};
\draw (ss3) ellipse (0.60 and 0.60);


\end{tikzpicture}
}

\newcommand{\acAb}{
  \begin{tikzpicture}[scale=0.7, vertices/.style={draw, fill=black, circle, inner sep=1.5pt}]

\node [vertices] (31) at (-0.8,0) {};
\node [vertices] (32) at (-0.2,0) {};
\node [vertices] (33) at (-0.5,0.7) {};

\node[anchor = east] at (31) {$a\,$ };
\node[anchor = west] at (32) {$\, c$};
\node[anchor = east] at (33) {$b \hskip 1.5mm $};
\draw (31)--(33);
\draw (32)--(33); 

\end{tikzpicture}
}


\begin{figure}
\begin{center}
\begin{tikzpicture}[inner sep=1pt, scale=0.6]

\coordinate (a) at (0,0) ; 
\coordinate (b) at (3, 1.71);  
\coordinate (c) at (3, 5.15);  
\coordinate (bc) at (6, 3.43);
\coordinate (d) at (0, 6.86);  
\coordinate (e) at (-3, 5.15);  
\coordinate (f) at (-3, 1.71);  

            \draw[black, thick, fill=tol1, fill opacity = 0.1] 
            (a) -- (bc) -- (d) -- (e) -- (f) -- (a);
            \draw[black, dotted] 
            (b) -- (c);
            \draw[black, fill] (a)     circle (2pt);
            \draw[black, fill] (bc)     circle (2pt);
            \draw[black, fill] (d)     circle (2pt);
            \draw[black, fill] (e)     circle (2pt);
            \draw[black, fill] (f)     circle (2pt);

\node[] at (a) [anchor = north east] {$(1,2,3)$};
\node[] at (bc) [anchor = north west] {$(3,0,3)$};
\node[] at (d) [anchor = south west] {$(3,2,1)$};
\node[] at (e) [anchor = south east] {$(2,3,1)$};
\node[] at (f) [anchor = north east] {$(1,3,2)$};
      
\end{tikzpicture}
\quad
\begin{tikzpicture}[inner sep=1pt, scale=0.6]

\coordinate (a) at (0,0) ; 
\coordinate (b) at (3, 1.71);  
\coordinate (c) at (3, 5.15);  
\coordinate (bc) at (6, 3.83);  
\coordinate (d) at (0, 6.86);  
\coordinate (e) at (-3, 5.15);  
\coordinate (f) at (-3, 1.71);  

\coordinate (x) at (-1.5,0.855) ; 
\coordinate (y) at (3, 1.71);  
\coordinate (t) at (3, 5.15);  
\coordinate (u) at (-1.5, 6.005);  
\coordinate (v) at (-3, 3.43);  

\coordinate (s) at (0, 3.43);

            \draw[black, thick, fill=tol1, fill opacity = 0.1] 
            (a) -- (bc) -- (d) -- (e) -- (f) -- (a);
            \draw[black, fill] (a)     circle (2pt);
            \draw[black, fill] (bc)     circle (2pt);
            \draw[black, fill] (d)     circle (2pt);
            \draw[black, fill] (e)     circle (2pt);
            \draw[black, fill] (f)     circle (2pt);

\node[] at (a) [anchor = north east] {$\cba$ \hskip 5mm {}};
\node[] at (bc) [anchor = north west] {$\acAb$ \hskip 5mm {}};
\node[] at (d) [anchor = south west] {$\abc$ \hskip 5mm {}};
\node[] at (e) [anchor = south east] {$\bac$ \hskip 5mm {}};
\node[] at (f) [anchor = north east] {$\bca$ \hskip 5mm {}};
      
\node[] at (x) [anchor = north east] {$\bcsa$};
\node[] at (y) [anchor = north west] {$\csab$};
\node[] at (t) [anchor = south west] {$\asbc$ };
\node[] at (u) [anchor = south east] {$\absc$};
\node[] at (v) [anchor = east] {$\bsac$};

\node[] at (s) {$\abcss$};
\end{tikzpicture}
          \end{center}
         \caption{The pentagon and the preorders associated to its faces}
        \label{fig:togenperm-conform}
    \end{figure}
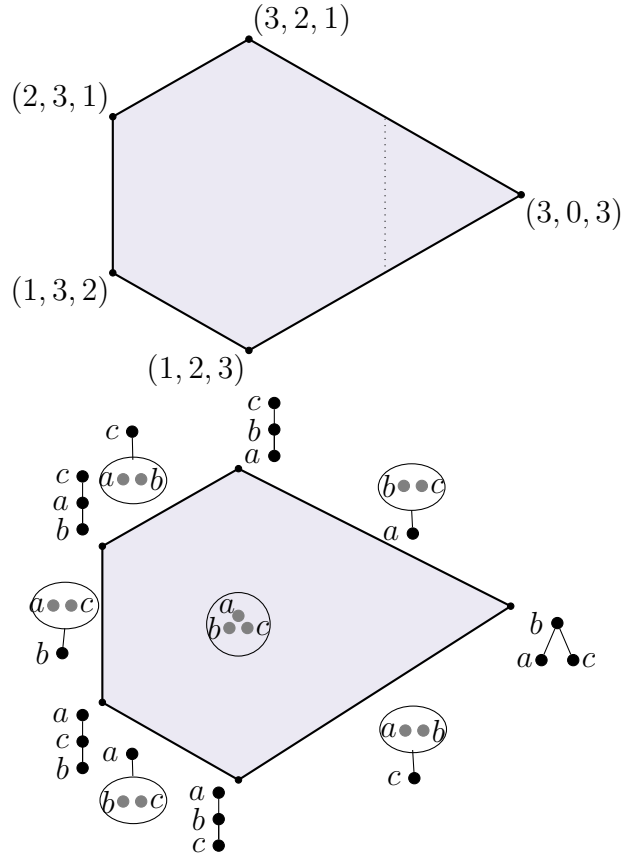    

Here are some further basic facts:
\begin{proposition} \label{pro:egp-dim}
  Let $z : \pow(I) \pil \hRR$ be a submodular function.
  \begin{itemize}
  \item[a.] For a preorder $P$, $\egperm(\low_P)$ is a cone, and $k(P)$ is
    its dual cone.
  \item[b.] The recession cone of $\egperm(z)$ is the extended generalized
    permutahedron $\egperm(\low_{\pre(z)})$.
  \item[c.] The dual cone of a recession cone is the cone of bounded
    directions, which here is the braid cone $k(\pre(z))$ of Fact
    \ref{fact:egp-braid}.
  \item[d.] The minimal faces of  $\egperm(z)$ have dimension
    $|I| - |{\mathbf b}(\pre(z))|$.
  \end{itemize}
\end{proposition}
  
  \begin{proof}
 a. When $z = \low_P$, then $\egperm(z)$ is defined by
  $x_I = 0$ and $x_S\leq 0$ for down-sets $S$ of $P$. The dual cone
  is then spanned by the directions $1_S$ and the linear space generated
  by $1_I$. This is precisely $k(P) = \Hom(P^\op, \RR)$.

  b. $\egperm(z)$ is defined by $x_S \leq z(S)$ for $S$ down-sets of $\pre(z)$,
  together with $x_I = z(I)$. The recession cone is then defined
  by $x_S \leq 0$ for all down-sets $S$ of $\pre(z)$, together with $x_I = 0$.

  c. is immediate from a. and b.


  d. The maximal dimensions of the normal cones of faces are
  $|{\mathbf b}(\pre(z))|$,
  by Proposition \ref{pro:braid-kp}.
\end{proof}

\subsection{Faces}

For a preorder $P$ and down-set $A$, denote by $P_A$ the restriction
of $P$ to $A$, and by $P_{/A}$ the restriction to $I \backslash A$. 

\begin{proposition}
  Let $F$ be a face of $\Pi$ with normal cone $k(P)$.
  Let $A$ be a down-set of $P$, so
  \[F \sus \Pi_A \times \Pi_{/A} = \egperm(z_{|A}) \times \egperm(z_{/A}). \]
  Then $F = F_A \times F_{/A}$ with $k(P_A)$ the normal cone of $F_A$,
  and $k(P_{/A})$ the normal cone of $F_{/A}$.
\end{proposition}

\begin{proof}
  Let $S \sus A$ be such that $1_S$ considered on $\Pi_A$ has
  maximal value on $F_A$. This value is $z_{|A}(S) = z(S)$.  
  Then also $1_S$ has maximal value $z(S)$ on $\Pi$.
  For $x \in F$ then  $(x_u)_{u \in A} \in F_A$.
  The value of $1_S$ on $x$ is then $x_S = z(S)$, the maximal value.
  Hence $1_S$ is in the normal cone of $F$, and so 
  $S$ is a down-set of $P_A$. 
  Conversely, any down-set $S$ of $P_A$ is also a down-set of $P$,
  and so gives a bounded direction $1_S$ for $F$. Thus the normal cone
  of $F_A$ is $k(P_A)$.
  
  Given $T \sus I \backslash A$ such that $1_T$ considered on
  $\Pi_{/A}$ has maximal value on $F_{/A}$. This value is $z_{/A}(T) =
  z(A \cup T) - z(A)$. The maximal value of $1_{A \cup T}$ on
$\Pi$ is $z(A \cup T)$. 
For $x \in F$ then $x_A = z(A)$ and $(x_u)_{u \in I \backslash A}$ is in $F_{/A}$.
Hence the value of $1_T$ on this point is $z(A \cup T) - z(A)$.
The value of $1_{A \cup T}$ on $x$ is then $z(A \cup T)$, the maximal value. 
  Then $1_{A \cup T}$ is in the normal cone of $F$ and so
    is a down-set of $P$. Whence $T$ is a down-set of $P_{/A}$.
    Conversely, given a down-set $A \cup T$ of $P$, so $1_{A \cup T}$
    is a bounded direction for $F$, then $1_T$ is a bounded direction for
    $F_{/A}$. In conclusion, $k(P_{/A})$ is the normal cone of $F_{/A}$.
    \end{proof}

\begin{lemma} \label{lem:egp-PAB}
  Let the face $F$ have normal cone $k(P)$, and suppose $P$ is a disjoint
  union of preorders $P = P_A \sqcup P_B$ with $I = A \sqcup B$.
  Then $z_{|A}$ and $z_{/B}$ are equal.
\end{lemma}

\begin{proof}
  We have
  \[ (\Pi)_{1_A} = \Pi_A \times \Pi_{/A}, \quad (\Pi)_{1_B} = \Pi_{/B}
    \times \Pi_B, \]
  and both contain $F$. Then for $x \in F$ we have
  \[ x_A = z(A), \quad x_B = z(B) \]
  and so
  \[z(I) = x_I = x_A + x_B = z(A) + z(B). \]
 Then, if $x \in P$ and $1_A$ takes maximal value $x_A = z(A)$, then
 since $x_B \leq z(B)$ and $x_B + x_A = x_I = z(I)$, we get $x_B = z(B)$,
 and so also $1_B$ takes maximum value on $x$. Thus
 $(\Pi)_{1_A} = (\Pi)_{1_B}$ and so
 \[ \egperm(z_{|A}) = \Pi_A = \Pi_{/B} = \egperm(z_{/B}). \]
\end{proof}

\begin{proposition} Let the face $F$ have normal cone $k(P)$ and 
  let $A \sus B \sus I$. Then $z_{B/A} := (z_{|B})_{/A}$ depends
  only on the set difference $C = B/A$ and not on $B$ and $A$.
Thus we shall write $z_C := z_{B/A}$. 
\end{proposition}

\begin{proof}
  Given the convex set $C$. Let $B$ be the smallest down-set containing $C$,
  and $A = B \backslash C$. Suppose also $C = B^\prime \backslash A^\prime$.
  Then $A^\prime = A \cup S$ and $B^\prime = B \cup S$, with $S \cap C
  = \emptyset$.
  Then $P_{B^\prime \backslash A}$ is the disjoint union of posets
  $P_S \sqcup P_C$, and using Lemma \ref{lem:egp-PAB}:
  \[ z_{B \backslash A} = (z_{B^\prime/A})_{|C} = (z_{B^\prime /A})_{/S} =
    z_{B \cup S \backslash A \cup S}. \] 
\end{proof}

The following describes the submodular functions corresponding
to the faces of $F$ of an extended general permutahedron $\egperm(z)$. 

\begin{theorem} \label{thm:egp-face}
  Let  $F$ be a face of $\egperm(z)$ with $k(P)$
  its normal cone. 
  Then $F$ is the extended generalized permutahedron with
  submodular function $z_P := \prod_{C \in {\mathbf b}(P)} z_C$.
\end{theorem}

\begin{proof}
  If $P = C$, the coarse topology, then $F = \egperm(z)$.
  Otherwise let $B$ be a non-trivial down-set of $P$. Then
  $F \sus \egperm(z_{|B}) \times \egperm(z_{/B})$, and we may write
  $F = F_B \times F_{/B}$.
  The normal cone of $F_B$ is $P_B$. By induction, the submodular function
  of $F_B$ is $\prod_{C \in {\mathbf b}(P_B)} z_C$. Correspondingly the submodular
  function of $F_{/B}$ is $\prod_{C \in {\mathbf b}(P_{/B})} z_C$.
\end{proof}

\subsection{Cones} \label{subsec:egp-cone}

\begin{proposition} \label{pro:mod-desc}
  A submodular function $z$ is modular iff $\egperm(z)$ is an affine cone.
  Then if $P = \pre(z)$, $\egperm(z)$ is a translation of
  $\egperm(\low_P)$. 
\end{proposition}

\begin{proof} $\egperm(z)$ is defined by $x_S \leq z(S)$ for the down-sets
  of $\pre(z)$. 
  Assume $\egperm(z)$ is an affine cone. There is then a minimal face, defined by
  making equalities of the defining inequalities. So
  for $x$ in this face we have $x_S = z(S)$. It then follows that
  \begin{equation} \label{eq:egp-zST}
    z(S \cup T) + z(S \cap T)  = x_{S \cup T} + x_{S \cap T} =
    x_S + x_T = z(S) + z(T). \end{equation}

  Assume $z$ is modular, defined by the function $\nu : {\mathbf b}(P)
  \pil \RR$ in Proposition \ref{pro:submod-mod}.
  The system of equations $x_{\downarrow C} = z(\downarrow C)$ has a solution
  since it is in triangular form.
  It is then straightforward to see by \eqref{eq:egp-zST}
  that $x_S = z(S)$ for each $S$ with $z(S) < \infty$.
  So this set of equalities has a solution, so $\egperm(z)$
  has a unique minimal face, and so is an affine cone.

  There is natural map $I \pil {\mathbf b}(P)$ and so $v : {\mathbf b}(P)
  \pil \RR$
  induces a map $y : I \pil \RR$,
  a point in $\RR^I$. Then $\egperm(z)$ is the translation of
  $\egperm(\low_P)$ by this point.
  \end{proof}

  \begin{definition} For a submodular function $z$,
    let the face $F$ of $\egperm(z)$ have normal cone $k(P)$.
  Define the submodular function $z^P$ by:
  \[ z^P(S):= \begin{cases} z(S), & S \text{ a down-set of } P, \\
                \infty, & \text{ if not.}
              \end{cases}
            \]
          \end{definition}

            
\begin{theorem} \label{thm:egp-cone} Let $z$ be submodular and 
$F$ a face of $\egperm(z)$ with $k(P)$ as its normal cone.

 \begin{itemize}
  \item[a.] $z^P$ is a modular function, and $\pre(z^P) = P$,
  \item[b.] The tangent cone $\cone_F(\egperm(z))$ is the extended generalized permutahedron
    $\egperm(z^P)$. 
 \end{itemize}
\end{theorem}

\begin{proof} For part b, in addition to the inequalities defining $\egperm(z)$,
  the face $F$ is defined by  $x_A = z(A)$ for down-sets $A$ of $P$,
  by Proposition \ref{pro:egp-normalcone}.
  By Lemma \ref{lem:polyh-conea}, $\cone_F(\egperm(z))$ is then defined by the
  inequalities $x_A \leq z(A)$ for down-sets $A$ of $P$.
  This is precisely the cone $\egperm(z^P)$.
  Part a. follows now by Proposition \ref{pro:mod-desc} b.
\end{proof}


\renewcommand{\bca}{
  \begin{tikzpicture}[scale=0.7, vertices/.style={draw, fill=black, circle, inner sep=1.5pt}]

\node [vertices] (31) at (-0.5,0) {};
\node [vertices] (32) at (-0.5,0.5) {};
\node [vertices] (33) at (-0.5,1) {};
\node (3p) at (-1, 0.5) {};

\node[anchor = east] at (31) {$b$ \hskip 4mm {}};
\node[anchor = east] at (32) {$c$ \hskip 4mm {}};
\node[anchor = east] at (33) {$a$ \hskip 4mm {}};
\draw (31)--(32);
\draw (31)--(33);
\node[anchor = east] at (3p) {$P\!:  \, $ \hskip 4mm {}};

\end{tikzpicture}
}

\newcommand{\pbcsa}{
  \begin{tikzpicture}[scale=0.9, vertices/.style={draw, fill=black, circle, inner sep=1.5pt}]

\node [vertices, gray] (31) at (-0.33,-0.2) {};
\node [vertices, gray] (32) at (-0.67,-0.2) {};
\node [vertices] (33) at (-0.5,0.7) {};
\node (3pp) at (-1, 0.3) {};

\node[anchor = west] at (31) {$\, c$ };
\node[anchor = east] at (32) {$b \,$ };
\node[anchor = east] at (33) {$a$ \hskip 0.1pt {}};

\node (ss3) at (-0.5, -0.2) {};
\draw (ss3) ellipse (0.58 and 0.42);
\node[anchor = east] at (3p) {$P\!: \, \, $ \hskip 4mm {}};


\draw (33)--(-0.57, 0.20);

\end{tikzpicture}
}

\renewcommand{\bcsa}{
  \begin{tikzpicture}[scale=0.9, vertices/.style={draw, fill=black, circle, inner sep=1.5pt}]

\node [vertices, gray] (31) at (-0.33,-0.2) {};
\node [vertices, gray] (32) at (-0.67,-0.2) {};
\node [vertices] (33) at (-0.5,0.7) {};

\node[anchor = west] at (31) {$\, c$ };
\node[anchor = east] at (32) {$b \,$ };
\node[anchor = east] at (33) {$a$ \hskip 0.1pt {}};

\node (ss3) at (-0.5, -0.2) {};
\draw (ss3) ellipse (0.58 and 0.42);


\draw (33)--(-0.57, 0.20);

\end{tikzpicture}
}

\renewcommand{\bsac}{
  \begin{tikzpicture}[scale=0.9, vertices/.style={draw, fill=black, circle, inner sep=1.5pt}]

\node [vertices] (31) at (-0.5,-0.2) {};
\node [vertices, gray] (32) at (-0.33,0.7) {};
\node [vertices, gray] (33) at (-0.67,0.7) {};

\node[anchor = east] at (31) {$b$ \hskip 0.1pt {}};
\node[anchor = west] at (32) {$\, c${}};
\node[anchor = east] at (33) {$a\, $};

\node (ss3) at (-0.5, 0.7) {};
\draw (ss3) ellipse (0.58 and 0.42);

\draw (31)--(-0.57, 0.27);

\end{tikzpicture}
}

\renewcommand{\abcss}{
  \begin{tikzpicture}[scale=0.9, vertices/.style={draw, fill=black, circle, inner sep=1.5pt}]

\node [vertices, gray] (31) at (-0.33,-0.2) {};
\node [vertices, gray] (32) at (-0.67,-0.2) {};
\node [vertices, gray] (33) at (-0.5, 0.03) {};

\node[anchor = west] at (31) {$\, c$ };
\node[anchor = east] at (32) {$b \,$ };
\node[anchor = south east] at (33) {$a$};

\node (ss3) at (-0.5, -0.13) {};
\draw (ss3) ellipse (0.55 and 0.55);


\end{tikzpicture}
}

      
\begin{figure}
\begin{center}
\begin{tikzpicture}[inner sep=1pt, scale=0.6]

\coordinate (a) at (0,0) ; 
\coordinate (b) at (3, 1.71);  
\coordinate (c) at (3, 5.15);  
\coordinate (d) at (0, 6.86);  
\coordinate (e) at (-3, 5.15);  
\coordinate (f) at (-3, 1.71);  

\coordinate (x) at (-1.5,0.855) ; 
\coordinate (y) at (1.5, 0.855);  
\coordinate (z) at (3, 3.43);  
\coordinate (t) at (1.5, 6.005);  
\coordinate (u) at (-1.5, 6.005);  
\coordinate (v) at (-3, 3.43);  

\coordinate (s) at (2, 4.50);

\coordinate (fo) at (-3, 7.2);
\coordinate (fh) at (3, -1.71);
\coordinate (fu1) at (7.4, 2.565);
\coordinate (fu2) at (0.74, 9.8);

\draw[white, fill=tol1, fill opacity = 0.1] 
            (fh) -- (fu1)  -- (fu2) -- (fo)--(f);

            \draw[black, dotted, fill=tol1, fill opacity = 0.09] 
            (a) -- (b)  -- (c) -- (d) -- (e) -- (f) -- (a);
            \draw[black, fill] (a)     circle (2pt);
            \draw[black, fill] (b)     circle (2pt);
            \draw[black, fill] (c)     circle (2pt);
            \draw[black, fill] (d)     circle (2pt);
            \draw[black, fill] (e)     circle (2pt);
            \draw[black, fill] (f)     circle (2pt);

\draw[black, very thick]
(f) -- (fo);

\draw[black, very thick]
(f) -- (fh);

\node[] at (f) [anchor = east] {$\bca\, \,\,  $ \hskip 10mm {}};
      
\node[] at (a) [anchor = north east] {$\bcsa$};
\node[] at (e) [anchor = east] {$\bsac$};

\node[] at (s) {$\abcss$};
\end{tikzpicture}
          \end{center}
          \caption{The cone at a vertex of the permutahedron, and
            the preorders of the four faces of this cone}
        \label{fig:cone-vertex-toperm}
      \end{figure}
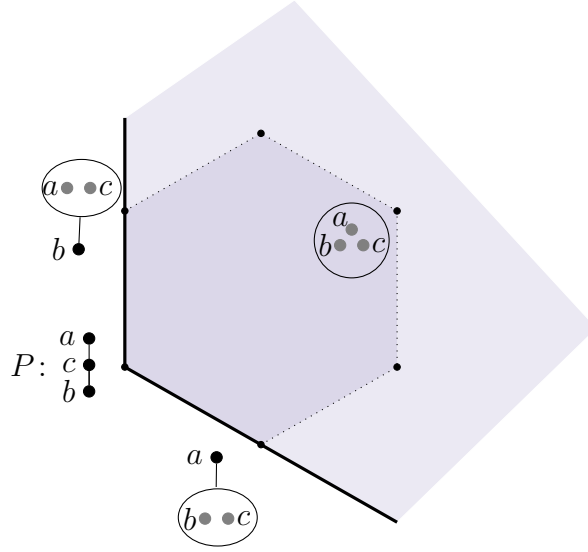

\ignore{
\begin{figure}
\begin{center}
\begin{tikzpicture}[inner sep=1pt, scale=0.8]

\coordinate (a) at (0,0) ; 
\coordinate (b) at (3, 1.71);  
\coordinate (c) at (3, 5.15);  
\coordinate (d) at (0, 6.86);  
\coordinate (e) at (-3, 5.15);  
\coordinate (f) at (-3, 1.71);  

\coordinate (x) at (-1.5,0.855) ; 
\coordinate (y) at (1.5, 0.855);  
\coordinate (z) at (3, 3.43);  
\coordinate (t) at (1.5, 6.005);  
\coordinate (u) at (-1.5, 6.005);  
\coordinate (v) at (-3, 3.43);  

\coordinate (s) at (2, 4.50);

\coordinate (fo) at (-6, 3.42);
\coordinate (fh) at (3, -1.71);
\coordinate (fu1) at (7.4, 2.565);
\coordinate (fu2) at (0.74, 9.8);

\draw[white, fill=tol1, fill opacity = 0.1] 
            (fh) -- (fu1)  -- (fu2) -- (fo)--(f);

            \draw[black, dotted, fill=tol1, fill opacity = 0.09] 
            (a) -- (b)  -- (c) -- (d) -- (e) -- (f) -- (a);
            \draw[black, fill] (a)     circle (2pt);
            \draw[black, fill] (b)     circle (2pt);
            \draw[black, fill] (c)     circle (2pt);
            \draw[black, fill] (d)     circle (2pt);
            \draw[black, fill] (e)     circle (2pt);
            \draw[black, fill] (f)     circle (2pt);

\draw[black, very thick]
(f) -- (fo);

\draw[black, very thick]
(f) -- (fh);

      
\node[] at (x) [anchor = north east] {$\pbcsa$};

\node[] at (s) {$\abcss$};
\end{tikzpicture}
          \end{center}
          \caption{The cone $\Pi(z^P)$ for $P$,
            and the conforming preorders of the two faces of this cone}
        \label{fig:cone-edge-toperm}
      \end{figure}
}

By Proposition \ref{pro:egp-dim}, $k(P)$ and $\egperm(\low_P)$ are dual cones.
  The cone $\egperm(\low_P)$ is described with generators in
  \cite[Thm. 3.38]{Fuji}. Section 3.4 of \cite{AA2023}
  considers it in detail and calls it a preposet cone.
  We give in Proposition \ref{pro:egp-dualcone} below an efficient argument for
Proposition 3.4.1 and Lemma 3.4.3 in 
  \cite{AA2023}.

  Let $F$ be a face of $\egperm(\low_P)$. Its {\it normal cone} is $k(Q)$
  for some $P \btl Q$ by \eqref{eq:braid-galois},
  i.e. a contraction $Q$ of $P$.
But as $F$ is a cone, it has also a {\it dual cone}:

\begin{proposition} \label{pro:egp-dualcone}
  The dual cone of $F$ is $k(R)$ where $R \wtl P$ corresponds
  to $P \btl Q$ via the bijection \eqref{eq:braid-galois}.
  Thus $F = \egperm(\low_R)$.
  
  The generators of the cone $F$ is the $e_j - e_i$ for
  $i \leq_R j$.
\end{proposition}

\begin{proof}
  We apply Theorem \ref{thm:egp-cone}b. above to $z = \low_P$ and
  the face $F$ with normal cone $k(Q)$, so $z^Q = \egperm(\low_Q)$.
  Then
  \[ \egperm(\low_P) - F = \cone_F(\egperm(\low_P)) = \egperm(\low_Q). \]
  We then have a diagram
\[ \xymatrix{
   k(Q) \ar[d]_{(-)^\vee} \ar[r]^-{\text{Galois}}
   & k(P) + (- k(Q)) \ar[d]^{(-)^\vee}\\
   \egperm(\low_Q) \ar[r]^-{\text{Galois}} & F, }
 \]
 where the horizontal arrows are the Galois connections for the
 cones $k(P)$ and $\egperm(\low_P)$ respectively, and the
vertical arrows are dual cones.
Due to the commutativity of the
diagram, we see that the dual cone of $F$ is           
is $k(P) + (-k(Q)) = k(P \wedge Q^{\op}) = k(R)$.
By Subsection \ref{subsec:duals},
the cone $F$ is then generated by the rays $e_j - e_i$ where $i \leq_R j$.
\end{proof}

\section{Cointeractions for extended generalized
  permutahedra}

\label{sec:cointer}

We relate the two bimonoids in species defined on
extended generalized permutahedra.
We show they are in cointeraction, but not in the sense
of \cite{Man2018},
the standard way for combinatorial Hopf algebras.
Rather, the cointeraction can be described as a measuring by a monoid,
as given in Appendix \ref{sec:meas-species}.

\medskip
Let $\egp(I)$ be the set of extended
generalized permutahedra in $\RR I$. This gives a species
\begin{equation} \label{eq:cointer-set} \egp : \setx \pil \set.
\end{equation}

\subsection{The bimonoids $\egp$ and $\smod$}
We want to consider comonoids in species. The natural thing is rather
to consider the above \eqref{eq:cointer-set} as a species:
\[ \egp : \setx \pil \setn. \]
This becomes a monoid in species (see Subsection \ref{subsec:species-setn}.
by the product of EGP's $\Pi_1 \times \Pi_2$.
By \cite[Section 1.4]{AA2023} this is also a comonoid in species with coproduct
\begin{equation*}
  \Delta : \egp(I) \pil
  \coprod_{{I = S \sqcup T}}
  \egp(S) \times \egp(T), \qquad
  \Pi \mapsto \sum_{1_S \text{ bounded on } \Pi} (\Pi_{|S}, \Pi_{/S}),
\end{equation*}

and counit

\begin{equation*} \epsilon : \egp \pil \benc, \quad \Pi \mapsto
\begin{cases} * \mapsto 1
     & I = \emptyset, \, \Pi = (0) \\
  * \mapsto 0, & I \neq \emptyset.
                 \end{cases}
\end{equation*}

Note that using the category $\setn$, introduced in the context
of bimonoids in species in \cite{F-CHA},
makes this a categorical comonoid. In contrast, the coproduct in
\cite[Section 1.1.2]{AA2023} for species in $\set$ must be ad hoc'ly defined.
(Comonoids in the categorical sense in the category of species in $\set$
do not exist, except in extremely special cases.)
With the product $\mu$ and the coproduct $\Delta$, this makes
$\egp$ into a bimonoid in species.

\medskip
By the correspondence between extended generalized permutahedra
and submodular functions, let
\[ \smod : \setx \pil \setn \]
be the species with $\smod(I)$ the submodular functions on $I$.
As $\smod(I)$ and $\egp(I)$ are in bijection, this makes $\smod$
into a bimonoid in species with product
\begin{align*}
  \mu: \smod(S) \times \smod(T) & \pil \smod(I)  \\
  (u,v) & \mapsto u \cdot v,
\end{align*}
and coproduct
\begin{equation*}
  \Delta : \smod(I)  \pil
  \coprod_{I = S \sqcup T}
    \smod(S) \times \smod(T), \qquad
  z  \mapsto \sum_{z(S) < \infty} (z_{|S}, z_{/S}).
\end{equation*}

\subsection{The subspecies of cones and modular functions}

Let $\egca : \setx \pil \setn$ be the sub-species of $\egp$ of
extended generalized permutahedra $\Gamma$  which are affine cones.
These are $\Gamma = \egperm(z)$ where $z$ is a modular function.

The coproduct $\delta_{\sfc}$ of \eqref{eq:cone-delta} now restricts to give
a comonoid for the Hadamard product:
\begin{align*}  \delta_{\egc} : \egca(I) & \pil \egca(I) \times \egca(I) \\
  \Gamma  & \mapsto \sum_{F \text{ face of } \Gamma} (F,\cone_F(\Gamma)).
\end{align*}
With the product $\mu$ and coproducts $\Delta$ and $\delta_{\egc}$,
this is a twisted double bialgebra in the language of \cite{Fo2019}.

Let $\mod$ be the subspecies of $\smod$, of modular functions.
By Theorems \ref{thm:egp-face} and \ref{thm:egp-cone}, the
comonoid on $\egca$ then corresponds to the coproduct
\begin{align*} \delta_{\mod} : \mod(I) & \pil \mod(I) \times \mod(I) \\
  z & \mapsto (z_P, z^P).
\end{align*}

\subsection{The comodule structure on $\egp$}
\label{subsec:EgpCom}
The comodule map $\delta_{\sfp}$ of \eqref{eq:poly-delta} gives a comodule
structure on $\egp$:
\begin{align*} \delta_{\egp} : \egp & \pil \egp \times \egca \\
  \Pi & \mapsto \sum_{F \text{ face of } \Pi} (F, \cone_F(\Pi)).
\end{align*}

On submodular functions, by Theorems \ref{thm:egp-face} and
\ref{thm:egp-cone} this is:
\begin{align*} \delta_{\smod} : \smod & \pil \smod \times \mod \\
  z & \mapsto (z_P, z^P).
\end{align*}

All the monoid and comonoid structures are now connected by the
following.

\begin{theorem}
  The comodule map $\delta_{\egp}$ is a measuring between the
  bimonoids $(\egp, \mu, \Delta)$ and $(\egca, \mu, \Delta)$
  by the monoid $(\egp, \mu)$.
\end{theorem}

\begin{proof}
First, that $\delta_{\egp}$ is a morphism of monoids is clear, since
$\delta_{\sfp}$ is this.
Secondly, we must show that the following diagram commutes:
\begin{equation} \label{eq:cointer-diagram}
  \xymatrix
  {\egp \ar[r]^{\delta_{\egp}} \ar[d]_{\Delta_{\egp}} &
    \egp \overset{H}\times \egca \ar[d]^{1 \times \Delta_{\egc}} \\
    \egp \overset{C}\times \egp \ar[d]_{\delta \times \delta} &
    \egp \overset{H}\times (\egca \overset{C} \times \egca)  \\
    (\egp \overset{H} \times \egca) \overset{C}\times
    (\egp \overset{H}\times \egca) \ar[r]^{1 \times \sigma \times 1}
    & (\egp \overset{C}\times \egp) \overset{H}
    \times (\egca \overset{C}\times \egca), \ar[u]_{\mu \times 1 \times 1}}
\end{equation}
where $\sigma$ is the natural twist map.
At the end (middle right), on $\egp(I) \times \egca(S) \times \egca(T)$,
this becomes as follows:
Following the maps to the left and bottom, $\Pi$ maps to
$(\Pi_{|S}, \Pi_{/S})$, then to the sum
\begin{equation} \label{eq:cointer-left}
  \sum_{F_1, F_2} (F_1 \times F_2, \cone_{F_1}(\Pi_{|S}), \cone_{F_2}(\Pi_{/S})).
  \end{equation}

Following the maps to the top and right, $\Pi$ maps to the sum
$\sum_{F} (F, \cone_F(\Pi))$ and then to
\begin{equation} \label{eq:cointer-right}
  \sum_F (F, (\cone_F(\Pi))_{|S}, (\cone_F(\Pi))_{/S}),
  \end{equation}
  where we sum over faces $F$ of $\Pi$ where $1_S$ takes a maximum value.
From Corollary \ref{cor:poly-w} we know that for a polyhedron $P$, a
direction $w$ and $F  \sus P_w$, then
\begin{equation} \label{eq:cointer-Pw}  \cone_F(P_w) = (\cone_F(P))_w.
\end{equation}
Here $P = \Pi, w = 1_S$, and
\[ F \sus (\Pi)_{1_S} = \Pi_S \times \Pi_{/S}. \]
Thus $F = F_1 \times F_2$, a face of $\Pi_S \times \Pi_{/S}$.
The left side of \eqref{eq:cointer-Pw} is then:
\[ \cone_F((\Pi)_{1_S}) = \cone_{F_1 \times F_2}(\Pi_S \times \Pi_{/S})
  = \cone_{F_1}(\Pi_S) \times \cone_{F_2}(\Pi_{/S}). \]
The right side of \eqref{eq:cointer-Pw} is:
\[ (\cone_F(\Pi))_{1_S} = (\cone_F(\Pi))_{|S} \times (\cone_F(\Pi))_{/S}. \]
Thus \eqref{eq:cointer-left} and \eqref{eq:cointer-right} become equal,
giving the diagram \eqref{eq:cointer-diagram}.
\end{proof}

\begin{remark}
  Recently L.Foissy, \cite{Foissy-boole}, has considered cointeractions
  for boolean functions $\pow(I) \pil \hele$. For a certain subclass of these,
  called rigid functions (actually a somewhat more general class),
  he has established a cointeraction in the
  classical sense.
\end{remark}

\appendix

\section{Measuring coalgebras and algebras}

\label{sec:measure}

Then notion of measuring coalgebras is a standard notion with very
nice properties and amenable to great generality. For material
on this, the reader may consult \cite{Ba-Univ, Ba-Comodule, Hy-Measure,
North}.
The analog dual notion of measuring algebra seems hardly to have been
considered. However this is the notion which seems right in our setting,
and we introduce it here.

\subsection{Measuring coalgebras}
We consider algebras and coalgebras over a field $k$. All
tensor products are over $k$. For vector spaces $V,W$, denote
by $[V,W]$ the set of linear maps $V \pil W$. 

Let $C$ be a coalgebra and $A,B$ algebras, and $\alpha : C \te A \pil B$
a linear map. The linear maps $[C,B]$ become an algebra under the
convolution product. The coalgebra $C$ {\it measures} the algebras $A$ and $B$,
if the linear map $A \pil [C,B]$ is an algebra homomorphism.

This is equivalent to the following two diagrams commuting:
\[\xymatrixcolsep{3pc}\xymatrix{C \te A \te A \ar[r]^-{\Delta_C \te 1_{A^{\te 2}}}
\ar[dd]^{1_C \te \mu_A}  &
(C^{\te 2}) \te (A^{\te 2}) \ar[r]^{\iso} & (C \te A)^{\te 2}
\ar[d]^{\alpha \te \alpha} \\ 
& & B \te B \ar[d]^{\mu_B} \\
C \te A \ar[rr]^{\alpha} & & B,
} \]

and
\[ \xymatrix{ C \te k \ar[d]^-{1_C \te \eta_A}
    \ar[r]^{\epsilon_C \te 1_k} & k \ar[d]^{\eta_B} \\
    C \te A \ar[r]^{\alpha} & B.}
\]

The idea is that $C$ parametrizes algebra homomorphisms $A \pil B$.
For fixed $A$ and $B$, one gets a category of coalgebras measuring
$A$ and $B$.

Instead of algebras $A$ and $B$ one may consider {\it coalgebras} $E$ and $F$.
For a linear map $\gamma : C \te E \pil F$, there is
also the notion of $C$ {\it measuring} $E$ and $F$. This is simply
equivalent to $\gamma$ being a coalgebra morphism.

\subsection{Measuring algebras}

We now consider the ``dual'' notion. Let $D$ be an algebra
and $E,F$ coalgebras, and $\beta : E \pil D \te F$ a linear
map. We say the algebra $D$ {\it measures} the coalgebras
$E$ and $F$, if the following diagrams commute:

\begin{equation} \label{eq:measure-EDF}
\xymatrixcolsep{3pc}\xymatrix{E \ar[d]^{\Delta_E} \ar[rr]^{\beta}& &  
      D \te F \ar[dd]^{1_D \te \Delta_F} \\
  E \te E  \ar[d]^{\beta \te \beta} & & \\     
 (D \te F)^{\te 2} \ar[r]^{\iso} &  D^{\te 2} \te F^{\te 2}
 \ar[r]^-{\hskip 2mm \mu_D \te 1_{F^{\te2}} \hskip 2mm} & D \te F \te F}
\end{equation}
and
\begin{equation}  \label{eq:measure-EDFeta}
  \xymatrix{E \ar[d]^{\epsilon_E}
    \ar[r]^{\beta} & D \te F\ar[d]^{1_D\te \epsilon_F} \\
   k \ar[r]^-{\eta_D \te 1_k} & D \te k.}
\end{equation}
The idea is that $D$ parametrizes coalgebra morphisms $E \pil F$.
For fixed $E$ and $F$ 
one gets a category of algebras measuring $E$ and $F$.

Again one may consider algebras $A,B$ (instead of coalgebras $E$ and $F$)
and linear maps $\delta : A \pil D \te B$.
The notion that $D$ measures $A,B$ is that $\delta$ is an
algebra morphism.

If $A$ and $B$ are bialgebras, we say the algebra $D$ {\it measures
  the bialgebras} $A$ and $B$ if it both measures $A,B$ as coalgebras
and as algebras.

\section{Cointeracting bialgebras} \label{sec:cointbi}

We see that the notion of comodule bialgebras \cite{Man2018}
can be seen as a measuring by an algebra.

\subsection{Comodule bialgebras}
Let $(B, \mu_B, \delta_B, \eta_B, \epsilon_B)$ and
$(A, \mu_A, \Delta_A, \eta_A, \epsilon_A)$ be bialgebras. Furthermore
we assume that $A$ is a left comodule over $B$, given by
$\delta : A \pil B \te A$.
Note that since $B$ is a bialgebra, $k$ and $A \te A$ are also a comodules
over $B$

We say $A$ is a {\it comodule bialgebra} over $B$ if all structural
maps of $A$: $\mu_A, \Delta_A, \eta_A, \epsilon_A$ are comodule maps.

\begin{proposition} \label{pro:cointer-measure}
Let $A$ and $B$ be a bialgebras and suppose 
$A$ is a left $B$-comodule for the map $\delta_A : A \pil B \te A$.
Then $A$ is a comodule bialgebra over the bialgebra $(B,\mu_B,\delta_B, \eta_B,
\epsilon_B)$
if and only if the algebra $(B,\mu_B,\eta_B)$ measures
the bialgebras $(A,\mu_A,\Delta_A)$ and $(A,\mu_A,\Delta_A)$.
\end{proposition}

\begin{proof}
  That the algebra structure maps $\mu_A : A \te A \pil A$
  and $\eta_A : k \pil A$ are
  morphisms of $B$-comodules, is equivalent to $A \pil B \te A$ being
  a morphism of algebras.
  Thus $B$ measures the algebras $(A,\mu)$ and $(A,\mu)$.
  That $\Delta : A \pil A \te A$ is a comodule morphism is
  the commutativity of the diagram
 \[\xymatrixcolsep{3pc} \xymatrix{A \ar[d]^{\Delta_A} \ar[rr]^{\delta}& &  
      B \te A \ar[dd]^{1_B \te \Delta_A} \\
 A \te A  \ar[d]^{\delta \te \delta}  & & \\    
  (B \te A)^{\te 2} \ar[r]^{\iso} &  B^{\te 2} \te A^{\te 2}
 \ar[r]^-{\mu_B\te 1_{A^{\te 2}}} & B \te A \te A }
\]
and that $\epsilon_A : A \pil k$ is a co-module morphism the commutativity of:
\[ \xymatrix{A \ar[d]^{\epsilon_A}
    \ar[r]^{\delta} & B \te A \ar[d]^{1_B\te \epsilon_A} \\
   k \ar[r]^{\eta_B} & B \te k.}
\]
But these diagrams are the diagrams \eqref{eq:measure-EDF} and
\eqref{eq:measure-EDFeta} with $E,D,F$ replaced with $A,B,A$.
\end{proof}

\subsection{Double bialgebras}

We may specialize to $A = B$, the multiplications $\mu_A = \mu_B$
and the comodule map $\delta : B \pil B \te B$ being the same as the coproduct
$\delta_B$ of $B$. In this case, the pair of bialgebras
\[ (B, \mu_B, \Delta_B, \eta_B, \epsilon_\Delta), \quad (B, \mu_B, \delta_B,
  \eta_B,
  \epsilon_\delta), \]
is called a {\it double bialgebra}, \cite{Fo2022}.

\subsection{Extending double bialgebras in two ways}
\label{subsec:ExtendTwo}
Now start with a double bialgebra $(B, \mu_B, \Delta_B, \delta_B, \eta_B,
\epsilon_\Delta, \epsilon_\delta)$.
We consider a bialgebra $(A,\mu_A,\Delta_A, \eta_A, \epsilon_A)$ which
extends the bialgebra $(B,\mu_B,\Delta_B, \eta_B, \epsilon_\Delta)$,
i.e. $B \sus A$ and $\mu_A, \Delta_A$ and $\epsilon_A$ restrict to
$\mu_B, \Delta_B$ and $\epsilon_\Delta$.

There are two ways to extend the coproduct map
$\delta_B : B \pil B \te B$ to a comodule map.

\begin{itemize}
\item[1.] $A$ is a left $(B,\delta_B)$-comodule:
  $\delta : A \pil B \te A$.
  This is the case considered in Proposition \ref{pro:cointer-measure}.
  Assuming that the algebra $(B,\mu_B)$
  measures the bialgebras $(A,\mu_A, \Delta_A)$ and $(A,\mu_A, \Delta_A)$,
  is equivalent to $(A,\mu_A, \Delta_A) $
  being a comodule bialgebra over $(B,\mu_B,\delta_B)$.
\item[2.] Or, $A$ is a right $(B,\delta_B)$-comodule:
  $\delta : A \pil A \te B$.
  Our assumption will be that the algebra  $(A,\mu_A)$ (in $A \te B$)
  measures the bialgebras
  $(A,\mu_A, \Delta_A)$ and $(B,\mu_B, \Delta_B)$.
  This is the case we encounter in the present article.
\end{itemize}

\begin{remark}
  Note that in case 1, to make sense of that $B$ measures the coalgebras $A$
  and $A$,
  it is not necessary for $\delta_A$ to
  be a comodule map for the coalgebra $(B,\delta_B)$,
  nor for the bialgebra $(A, \mu_A, \Delta_A) $
  to be an extension of the bialgebra $(B, \mu_B, \Delta_B)$.

  The analog also holds for case 2. However in case 2, we should have
  a bialgebra $(B, \mu_B, \Delta_B)$, as we consider an algebra $(A,\mu_A)$
  which measures bialgebras. In case 1 the minimal requirement
  is that we have an algebra $(B,\mu_B)$.

\end{remark}


\section{Measuring of species}
\label{sec:meas-species}

The notions in Appendices \ref{sec:measure} and \ref{sec:cointbi}
may now be extended to species.
Let $D$ be a monoid in species and $E, F$ comonoids in species for
the Cauchy product. The monoid $D$ {\it measures} the comonoids $E$ and $F$,
if there is a morphism of species
$\beta: E \pil D \teh F$ such that the following diagram commutes:

\begin{equation*}
\xymatrixcolsep{3pc}\xymatrix{E \ar[d]^{\Delta_E} \ar[rr]^{\beta}& &  
      D \teh F \ar[dd]^{1_D \te \Delta_F} \\
  E \tec E  \ar[d]^{\beta \te \beta} & & \\     
 (D \teh F) \tec (D \teh F) \ar[r]^-{\iso} &  D^{\tec 2} \teh F^{\tec 2}
 \ar[r]^-{\mu_D \te 1_{F^{\te 2}}} & D \teh (F \tec F)}
\end{equation*}
and
\begin{equation*} 
  \xymatrix{E \ar[d]^{\epsilon_E}
    \ar[r]^{\beta} & D \teh F \ar[d]^{1_D\te \epsilon_F} \\
   \benh \ar[r]^{\eta_D \te 1} & D \teh \benh.}
\end{equation*}

If $E, F$ are bimonoids in species (for Cauchy-products) we have the 
notion of the monoid $D$ {\it measuring the bimonoids} $E$ and $F$.
We also have the notion for species of a comodule bimonoid $(A,\mu_A, \Delta)$
over a bimonoid $(B,\mu_B, \delta)$ in species, where $\mu_A, \mu_B$ and $\Delta_A$ are
given by the Cauchy product, and $\delta_B$ and the comodule map 
$\delta_A$ given by the Hadamard product. 
By specializing to $A = B$ we get the notion of double bimonoid
$(B,\mu,\Delta, \delta, \eta, \epsilon_\Delta, \epsilon_\delta)$.

\medskip
Starting from a double bimonoid $B$ in species,
we may have an extension bimonoid
$(A,\mu_A, \Delta_A)$ of $(B,\mu_B,\Delta_B)$.
We also assume $(A, \delta_A)$ is a comodule over the comonoid $(B,\delta_B)$
such that $\delta_A$ extends $\delta_B$. Lifting the settings
of Subsection \ref{subsec:ExtendTwo} to species, we could have:

\begin{itemize}
\item[1.]  $(A,\delta_A) $ is a {\it left} comodule
over $(B,\delta_B)$,
with $(B,\mu)$ measuring the bimonoids $(A,\mu,\Delta)$ 
and $(A,\mu, \Delta)$. This is equivalent to $(A,\mu,\Delta)$ being
a comodule bimonoid over $(B,\mu,\delta)$. 
\item[2.] Or, $(A,\delta_A)$ is a {\it right} comodule over $(B,\delta_B)$, with
$(A,\mu)$ measuring the bimonoids $(A,\mu,\Delta)$ and $(B,\mu,\Delta)$.
\end{itemize}
The latter is the setting we encounter in Subsection \ref{subsec:EgpCom}.

\bibliographystyle{amsplain}
\bibliography{biblio}

\end{document}